\documentclass[11pt, letterpaper]{amsart}
\usepackage{graphicx, amssymb, color}
\usepackage{hyperref}

\newcommand\F{\mbox{I\kern-2pt F}}

\newcommand\cE{{\mathcal E}}

\newcommand\cF{{\mathcal F}}

\newcommand\cL{{\mathcal L}}
\newcommand\cB{{\mathcal B}}

\newcommand\cM{{\mathcal M}}
\newcommand\cX{{\mathcal X}}

\newcommand\cR{{\mathcal R}}
\newcommand\cP{{\mathcal P}}

\newcommand\cS{{\mathcal S}}

\newcommand\cZ{{\mathcal Z}}
\newcommand\cW{{\mathcal W}}
\newcommand\R{{\mathbb R}}
\newcommand\e{{\varepsilon}}

\newtheorem{theo}{Theorem}[section]
\newtheorem{prop}[theo]{Proposition}
\newtheorem{lemm}[theo]{Lemma}

\theoremstyle{remark}
\newtheorem{rem}[theo]{Remark}

\newcommand\beq{\begin{equation}}
\newcommand\eeq{\end{equation}}
\newcommand\bea{\begin{eqnarray}}
\newcommand\eea{\end{eqnarray}}
\newcommand\bean{\begin{eqnarray*}}
\newcommand\eean{\end{eqnarray*}}

\begin{document}

\title[NA$_1$ and local martingale deflators]{No arbitrage and local martingale deflators}

\thanks{The research of Yuri Kabanov  is funded by the grant of the Government of Russian Federation  $n^\circ$ 14.A12.31.0007}

\keywords{Fundamental Theorem of Asset Pricing; arbitrage; viability; num\'eraire; local martingale deflator; $\sigma$-martingale}

\subjclass{91B70; 60G44}

\date{\today}

\author{Yuri Kabanov}%
\address{Yuri Kabanov: Laboratoire de Math\'ematiques, Universit\'e de
Franche-Comt\'e, 16 Route de Gray,\\
 25030 Besan\c{c}on, cedex,
France,  and 
International Laboratory of Quantitative Finance,
Higher  School of Economics, Moscow,
Russia}%
\email{youri.Kabanov@univ-fcomte.fr}%

\author{Constantinos Kardaras}%
\address{Constantinos Kardaras: Department of Statistics, London School of Economics,   10 Houghton Street, London, WC2A 2AE, England}%
\email{k.kardaras@lse.ac.uk}%

\author{Shiqi Song}%
\address{Shiqi Song: Universit\'e  d'Evry Val d'Essonne, Boulevard de France, 91037 Evry, cedex, France}%
\email{bachsuitepremier@gmail.com}%

\begin{abstract}
A supermartingale deflator (resp., local martingale deflator) multiplicatively transforms nonnegative wealth processes into supermartingales (resp., local martingales). The supermartingale num\'eraire (resp., local martingale num\'eraire) is the wealth processes whose reciprocal is a supermartingale deflator (resp., local martingale deflator). It has been established in previous literature that absence of arbitrage of the first kind (NA$_1$)  is equivalent to existence of the supermartingale num\'eraire, and further equivalent to existence of a strictly positive local martingale deflator; however, under NA$_1$, the local martingale num\'eraire may fail to exist. In this work, we establish that, under NA$_1$, any total-variation neighbourhood of the original probability has an equivalent probability under which the local martingale num\'eraire exists. This result, available previously only for single risky-asset models, is in striking resemblance with the fact that any total-variation neighbourhood of a separating measure contains an equivalent $\sigma$-martingale measure. The presentation of our main result is relatively self-contained, including a proof of existence of the supermartingale num\'eraire under NA$_1$. We further show that, if the L\'evy measures of the asset-price process have finite support, NA$_1$ is equivalent to existence of the local martingale num\'eraire with respect to the original probability.
\end{abstract}

\maketitle

\section*{Introduction} 

A central structural assumption in the mathematical theory of financial markets  is the existence of so-called \emph{local martingale deflators}, i.e., processes that act multiplicatively and transform nonnegative wealth processes into local martingales. Under the \emph{No Free Lunch with Vanishing Risk} (NFLVR) condition of \cite{DS, DS:98}, the density process of a local martingale (or, more generally, a $\sigma$-martingale) measure is a strictly positive local martingale deflator. However, strictly positive local martingale deflators may exist even if the market allows for free lunch with vanishing risk. One may expect that strictly positive supermartingale deflators, multiplicatively transforming nonnegative wealth processes into supermartingales, may exist even in a wider class of models. 

Of special interest is the case where there exists a (necessarily unique) strictly positive supermartingale deflator which is the reciprocal of a strictly positive wealth process. If the latter wealth process is taken as the num\'eraire, the prices of all assets expressed in its units become supermartingales. In this case, we shall say that the model admits the \emph{supermartingale num\'eraire}.    
In a completely analogous way, we define the stronger notion of a \emph{local martingale num\'eraire}; in view of Fatou's lemma, local martingale num\'eraires are, in particular, supermartingale num\'eraires.

Since, loosely speaking, a supermartingale is a process ``decreasing in the mean,'' wealth processes that are supermartingale num\'eraires are expected to have certain optimality properties. Indeed, there is an extensive body of literature studying, under various levels of generality, supermartingale (or local martingale) num\'eraires in conjunction with  closely related concepts, such as log-optimal portfolios, optimal growth portfolios, etc.; see \cite{Algoet-Cover}, \cite{LW}, \cite{Becherer}, \cite{GK}, \cite{Rokhlin}, \cite{Hulley-Sch}, and references therein.

The relevant, weaker than NFLVR, no-arbitrage property connected to existence of supermartingale (or local martingale) num\'eraires was isolated by various authors under different names:  \emph{No Asymptotic Arbitrage of the 1st Kind} (NAA$_1$), \emph{No  Arbitrage of the 1st Kind} (NA$_1$), \emph{No Unbounded Profit with Bounded Risk} (NUPBR),  etc. It is not difficult to show that all these properties are in fact equivalent, even in a wider framework  than that of the standard semimartingale setting---for more information, see Appendix \ref{sec:arbitrage}. In the present paper we opt to utilise the economically meaningful formulation NA$_1$, defined as the property of the market to assign a strictly positive superreplication price to any non-trivial positive contingent claim.  

In the standard model studied here, the market is described by a $d$-dimensional semimartingale process $S$ giving the discounted prices of basic securities. In \cite{Kara-Kard} it was shown (in fact, under the potential presence of convex portfolio constraints) that the following statements are equivalent:\begin{enumerate}
	\item[(i)] Condition NA$_1$ holds.
	\item[(ii)] There exists a strictly positive supermartingale deflator.
	\item[(ii$'$)] The supermartingale num\'eraire exists.
\end{enumerate}
In \cite{Takaoka}, the previous this list of equivalent 
properties is complemented by
\begin{enumerate}
	\item[(iii)] There exists a strictly positive local martingale deflator. 
\end{enumerate}
There are counterexamples (see, for example, \cite{Takaoka}) showing that local martingale num\'eraire may fail to exist even when there is an equivalent martingale measure (and, in particular, condition NA$_1$ holds) in the market. Of course, such an example is possible only in the case of discontinuous asset-price process: it was shown in \cite{Ch-Str} that, for continuous semimartingales, amongst strictly positive local martingale deflators there exists one whose reciprocal is a  value process of some portfolio.

In the present paper, we add to the above list of equivalences a further property: 

\begin{enumerate}
 \item[(iv)] In any neighbourhood (in total variation distance) of the original probability, there exists an equivalent probability under which the strictly positive local martingale deflator exists.
\end{enumerate}
Moreover, we prove that if the L\'evy measures of the asset-price process have finite support, condition NA$_1$ is equivalent to existence of the strictly positive local martingale deflator under the original probability. 

Establishing the chain  $(iv)\Rightarrow (iii)\Rightarrow (ii)\Rightarrow (i)$ is rather straightforward, and well-known. Our main contribution, which is the proof of implication $(i)\Rightarrow (iv)$, implies the result shown in the aforementioned papers \cite{Kara-Kard} and \cite{Takaoka}. We note that implication   $(i)\Rightarrow (iv)$ was established in  \cite{Kard} in the  one-asset semimartingale  model, by a method not allowing a straightforward generalisation to the multi-dimensional case. 

The arguments of \cite{Takaoka} establishing implication $(i)\Rightarrow (iii)$ combine a change-of-num\'eraire techniques and a clever reduction to the Delbaen--Schachermayer Fundamental Theorem of Asset Pricing (FTAP); in effect, they are based on a result considered as one of the most difficult in the field of Mathematical Finance.  On the other hand, it was shown in \cite{Kard10} that one may actually obtain the FTAP in a somewhat simpler way once the existence of a strictly positive local martingale deflator is guaranteed. For this reason, it is quite desirable to find an alternative proof of existence of a strictly positive local martingale deflator (under NA$_1$ and, \emph{a fortiori}, under NFLVR). In the present paper, we succeed in doing so by using techniques related to semimartingale predictable characteristics, which has proved efficient in a number of related problems, see, e.g., \cite{K97}, \cite{KStra:00}, \cite{Kara-Kard}.   

In order to be as self-contained as possible, on the way to establishing implication $(i)\Rightarrow (iv)$ we obtain as an intermediate step a version of the implication $(i) \Rightarrow (ii')$, under the additional assumption that the function $x\mapsto |x|\wedge |x|^2$ is integrable with respect to the L\'evy measures of the process $S$. This assumption allows us to both simplify the reasoning, as well as a obtain a ``semi-explicit'' form of the supermartingale num\'eraire.   

The concluding step in the proof of implication $(i)\Rightarrow (iv)$ is based on  
the observation that the ratio of positive wealth processes can be represented as a wealth process in an auxiliary model with another semimartingale price process $\bar S$, i.e., with different dynamics of the basic securities. When the denominator of the ratio is the supermartingale num\'eraire, the original probability is a separating measure for $\bar S$ in the terminology of \cite{K97}. In this case, it was shown in \cite{DS:98} that there exists an equivalent measure under which $\bar S$ becomes a $\sigma$-martingale, implying that the previous relative wealth processes become local martingales, which is exactly the required property  $(iv)$.  

\subsection*{Structure of the paper}

In Section \ref{setting} we present our set-up, as well as formulate and discuss our main result. Section \ref{roadmap} contains preliminaries on semimartingales, and (in)equality conditions in terms of their local characteristics under which the ratio of two stochastic exponentials  is a local martingale, or simply a supermartingale. In Section \ref{sdeflator} we prove that the supermartingale num\'eraire exists under condition NA$_1$, coupled with a simplifying integrability assumption which makes the arguments more transparent. In contrast to \cite{Kara-Kard} we use a new method for  verification of integrability of a ``candidate'' portfolio strategy required to build the supermartingale num\'eraire.
In Section \ref{num_change} we show that a change of num\'eraire leads to a new model with a different price process  $\bar S$ for the basic securities.  When the supermartingale num\'eraire is used, the original probability $P$ is a separating measure for $\bar S$. By an important result in \cite{DS:98}, we infer that arbitrarily close  to $P$ there exist probability measures under which $\bar S$ is a $\sigma$-martingale, implying that the considered  supermartingale num\'eraire is a local martingale num\'eraire with respect to this new probability. 

We complement our presentation with a long Appendix. Although the paper is self-contained, knowledge of recent developments in arbitrage theory would certainly be beneficial; we hope that the interested reader, carrying knowledge from several sources with different levels of generality and language used, will appreciate the unified discussion of the main concepts and results collected in Appendix \ref{sec:arbitrage}. A short proof of the fact that in any total-variation neighbourhood of a separating measure there exists a $\sigma$-martingale measure, originally appearing in \cite{DS:98}, is given in Appendix \ref{existsigma}. Appendix \ref{boundSE} contains a structured presentation of certain results from \cite{Kara-Kard} relating boundedness  in probability of sets of stochastic exponentials with boundedness of their arguments. Finally, and in order to avoid disruption of the flow of the main line of our proof, we defer to  Appendix \ref{LLN} seemingly new  ``laws of large numbers'' for sequences of stochastic integrals with truncated integrands.

\section{Framework and Main Result}
\label{setting}

\subsection{The set-up}
In all that follows, we fix $T \in (0, \infty)$ and work on a filtered probability space $(\Omega,\cF,{\bf F}=(\cF_t)_{t \in [0, T]},P)$ satisfying the usual conditions. Unless otherwise explicitly specified, all relationships between random variables are understood in the $P$-a.s. sense, and all relationships between stochastic processes are understood modulo $P$-evanescence. (There will be other notions of equality between processes that will be used.)

Let $S=(S_t)_{t \in [0, T]}$ be a $d$-dimensional semimartingale. We denote by $L(S)$ the set of $S$-integrable  processes, i.e., the set of all $d$-dimensional predictable processes for which the stochastic integral $H\cdot S$ is defined.  We stress that we consider general vector stochastic integration.

An integrand $H\in L(S)$ such that $x + H\cdot S \ge 0$ holds for some $x\in \R_+$ will be called \emph{$x$-admissible}. We introduce the set of semimartingales  
$$
\cX^x := \{H\cdot S : \  H\ \hbox{is $x$-admissible integrand}\}
$$ 
and denote  $\cX^x_>$ its subset formed by processes $X$ such that $x + X > 0$ and $x + X_- > 0$.  These sets are invariant under equivalent changes of the underlying probability measure. Define also the sets of random variables $\cX^x_T:=\{X_T:\ X\in \cX^x\}$.

For $\xi \in L^0_+$, we define 
$$
\bar x(\xi):=\inf \{x\in \R_+: \  \hbox{there exists $X\in \cX^x$ with}\ x +X_T
\ge \xi\},
$$  
with the standard convention  $\inf \emptyset =\infty$.

The previous definitions are abstract. In the special context of financial modelling: 
\begin{itemize}
	\item $S$ represents the price process of $d$ liquid assets, discounted by a certain baseline security.
	\item With $H$ being $x$-admissible integrand, $x + H\cdot S$ is  the value  process of a self-financing  portfolio with the initial capital $x \geq 0$, constrained to stay nonnegative at all times. 
	\item $\xi \in L^0_+$ represents a contingent claim, and $\bar x(\xi)$ is its \emph{superreplication price} in the class of nonnegative wealth processes. 
\end{itemize}

\subsection{Main result}

Recall that $|P - \tilde {P}|_{TV} = \sup_{A \in \cF} |P(A) - \tilde{P} (A)|$ is the total variation distance between the probabilities $P$ and $\tilde {P}$ on $(\Omega, \cF)$. The following is the main result of the paper.

\begin{theo}
\label{main-I}

The following conditions are equivalent: 

\begin{enumerate}
	\item[(i)] $\bar x(\xi)>0$ for every $\xi\in L^0_+\setminus\{0\}$.

	\item[(ii)] There exists a strictly positive process $Y$ such that 
 the process $Y(1+X)$ is a supermartingale for every $X\in \cX^1$.

	\item[(iii)] There exists a strictly positive process $Y$ such that  
 the process $Y(1+X)$ is a local martingale for every $X\in \cX^1$.
  
 \item[(iv)] For any $\e>0$ there exist a probability measure $\tilde P\sim P$ with $| \tilde P-P |_{TV} <\e$
and $\tilde X = \tilde H \cdot S\in \cX^1_>$ such that $(1+X)/(1+ \tilde X)$ is a local $\tilde P$-martingale for every $X\in \cX^1$. 
\end{enumerate}
Moreover, if the L\'evy measures of $S$ are concentrated on a finite number of points, one can take $\tilde P=P$ in (iv) above. 
\end{theo}

\begin{rem}
It is straightforward to check that statements $(ii)$, $(iii)$, and $(iv)$ of Theorem \ref{main-I} are equivalent to the same conditions where ``for every $X\in \cX^1$'' is replaced by ``for every $X\in \cX^1_>$,'' a fact that will be used freely during its proof.
\end{rem}

Theorem \ref{main-I} is formulated in ``pure'' language of stochastic analysis. In the context of Mathematical Finance, the following interpretations regarding its statement should be kept in mind:
\begin{itemize}
	\item Condition (i) states that any non-trivial contingent claim $\xi\ge 0$ has a strictly positive superreplication price. This is referred to as condition NA$_1$ (No  Arbitrage of the 1st Kind); it is equivalent to the boundedness in probability of the set 
$\cX^1_T$, or, alternatively, to condition NAA$_1$ (No  Asymptotic Arbitrage of the 1st Kind)---see  Appendix A.

	\item The process $Y$ in statement $(ii)$ (resp., in statement $(iii)$) is called a strictly positive \emph{supermartingale deflator} (resp., \emph{local martingale deflator}).

	\item The process $1 + \tilde H \cdot S$ with the property in statement $(iv)$ is called the \emph{local martingale num\'eraire under the probability $\tilde{P}$}.
\end{itemize}

With the above terminology in mind, we may reformulate Theorem \ref{main-I} in a more appealing way as the equivalence of the following statements:

{\it
\begin{enumerate}
	\item[(i)] There is no arbitrage of the first kind in the market (NA$_1$).

	\item[(ii)] There exists a strictly positive supermartingale deflator.

	\item[(iii)] There exists a strictly positive local martingale deflator.
  
 \item[(iv)] In any neighbourhood (in total variation distance) of $P$ there exists a probability measure $\tilde P\sim P$ admitting a local martingale num\'eraire.
\end{enumerate}
}

\subsection{Proof of easy implications of Theorem \ref{main-I}}

The arguments establishing the implications $(iv) \Rightarrow (iii) \Rightarrow (ii) \Rightarrow (i)$ in Theorem \ref{main-I} are elementary and well known; however, for completeness of presentation,  we shall briefly reproduce them here.    

Assume statement $(iv)$, and in its notation set $\tilde Z := 1 / (1 + \tilde H  \cdot S)$, and let $Z$ be the density process of $\tilde P$ with respect to $P$. For any  $X\in \cX^1$, $\tilde Z(1+X)$ is a local $\tilde P$-martingale.  Hence, with $Y := Z \tilde Z$, the process $Y (1+X)$ is 
a local $P$-martingale. Implications $(iv) \Rightarrow (iii)$ follows.

Since a positive local martingale is a supermartingale, implication  $(iii) \Rightarrow (ii)$ is obvious. 

To establish implication $(ii) \Rightarrow (i)$,  suppose  that $Y$ is a strictly positive supermartingale deflator. It follows that $EY_T(1+X_T)\le 1$ holds for every $X\in \cX^1$. Hence, the set $Y_T(1+\cX^1_T)$ 
is bounded in $L^1$, and, \emph{a fortiori}, bounded in probability. Since $Y_T > 0$, the set $\cX^1_T$ is also bounded in probability. The latter property is equivalent to condition NA$_1$---see Appendix A. 

\smallskip

In contrast to the above elementary arguments, the proof of implication $(i) \Rightarrow (iv)$ is significantly more involved and requires intensive use of stochastic calculus. The rest of the paper, including Appendices \ref{existsigma} and \ref{boundSE}, is devoted to the proof of implication $(i) \Rightarrow (iv)$; the eventual argument is given in \S \ref{subsec:proof_final}. Before going through the proof, it may be helpful to consult Appendix \ref{sec:arbitrage} for a discussion of various  forms of condition NA$_1$, and relations with other concepts of arbitrage theory.

\section{Proof of Implication $(i) \Rightarrow (iv)$ of Theorem \ref{main-I}: Preliminaries}
\label{roadmap}

\subsection{Local characteristics of the discounted price-processes and simplifying hypotheses} \label{subsec:canon_decomp}

As a first step, we shall introduce notation, recall basic facts and make a certain reduction which will allow us to work under (slightly) simplifying hypotheses.

Following the development of ideas from \cite{JS:87}, let $(B^h, \, C, \, \nu)$ be the triplet of predictable characteristics of the semimartingale $S$, where the latter is written in its canonical form
$$
S=S_0+S^c+xh*(\mu-\nu)+x\bar h*\mu +B^h.
$$
Above, $h=I_{\{|x|\le 1\}}$  denotes the chosen truncation function and $\bar h := 1 - h = I_{\{|x|> 1\}}$, $S^c$ is a continuous 
local martingale and $B^h$ is a predictable process of bounded variation. 
It is convenient to work with a ``local'' form of the triplet. One can always choose a predictable increasing c\`adl\`ag process $A$ with $A_0=0$ and $A_T\le 1$
such that
$$B^h=b^h\cdot A,  \qquad \langle S^c\rangle=c\cdot A,  \qquad \nu(dt,dx)=dA_tK_{t}(dx),
$$
$b^h$ is predictable,
$K_{t}(dx) = K_{\omega,t}(dx)$ is a transition kernel from
$(\Omega\times \R_+,{\mathcal P})$
into
$({\R}^d,{\mathcal B}^d)$ with $\int (|x|^2\wedge 1)K_{\omega,t}(dx)<\infty$;
if $\Delta A_t(\omega)>0 $ then we have 
$\Delta A_t(\omega)K_{\omega,t}(\R^d)\le 1$, 
see \cite[II.2.9]{JS:87}.
Analogously to the theory of processes with independent increments, the measures $K_{\omega,t}$ will be referred to as \emph{L\'evy measures}. 

Let $\bar {\mathcal P}$ be the completion of the $\sigma$-algebra
${\mathcal P}$ with respect to the measure
$$
m(d\omega,dt):=P(d\omega)dA_t(\omega).
$$
We use the notation  
$$
K_{\omega,t}(Y):=\int Y(x)K_{\omega,t}(dx)
$$
and in places we omit $\omega$ or $(\omega,t)$ to alleviate heavy notation. 

\begin{lemm}
\label{special}
For any $\e>0$, there exists a probability measure $\tilde P\sim P$ with bounded density $d\tilde P/dP$ such that $| \tilde P-P |_{TV} \le \e$,
and $|x|^2*\mu_T \in L^1(\tilde P)$.
\end{lemm}

\begin{proof}
For all $n$, Consider the bounded probability density $Z^n_T:=c_n(1+n^{-1} |x|^2*\mu_T)^{-1}$, where $c_n$ is an appropriate normalizing constant. The probability measures $P^n=Z^n_T P$ for sufficiently large $n$ meet the requirements of the statement of the lemma.
\end{proof}

Condition NA$_1$ is invariant under equivalent changes of probability. Due to Lemma \ref{special} above, for the purposes of proving implication $(i)\Rightarrow (iv)$ of Theorem \ref{main-I} we may assume that $$E|x|^2*\nu_T=E|x|^2*\mu_T<\infty$$ holds. While such assumption would lead to slightly simpler arguments, we also wish to show that $\tilde P=P$ in the case where $m$-a.e. L\'evy measure $K_{\omega,t}(.)$ is concentrated on a finite number of points. For this reason, we shall work under the assumption
\begin{equation} \label{ass:sigma-special}
K(|x|^2\wedge |x|)<\infty, \quad m\text{-a.e.},
\end{equation}
which is weaker than $E|x|^2*\nu_T < \infty$. When the L\'evy measures have finite support, \eqref{ass:sigma-special} is always fulfilled; as we shall see later, no further probability measure change will be needed. In the case of general L\'evy measure structures, we use Lemma \ref{special} in order to ensure that \eqref{ass:sigma-special} is valid; additionally, another change of measure will be required during the concluding step of the proof in Lemma \ref{lem:final}.

Under the force of \eqref{ass:sigma-special}, note that the predictable process
\[
b := b^h + K\big(xI_{\{|x|>1\}}\big)
\]
is well defined. If $|b| \cdot A_T < \infty$, then $S$ is a special semimartingale and $B := b \cdot A$ is the finite variation process in its canonical decomposition. However, it could happen that $P (|b| \cdot A_T = \infty) > 0$; in effect, \eqref{ass:sigma-special} can be loosely interpreted as a ``$\sigma$-special'' property of $S$.


\subsection{Statement $(iv)$ and ratios of stochastic exponentials}
The set $1+\cX^1_>$ coincides with the set of stochastic exponentials of integrals with respect to $S$:  
$$
1+\cX^1_> =\{\cE(f\cdot S): \ f\in L(S),\ f\Delta S>-1\}.  
$$
Indeed, a stochastic exponential corresponding to the integrand $f$ such that $f\Delta S>-1$
is strictly positive, as is also its left limit, and satisfies the linear equation 
$$
\cE(f\cdot S)=1+\cE_-(f\cdot S)\cdot (f\cdot S)= 1+(\cE_-(f\cdot S) f)\cdot S. 
$$
Thus, $\cE(f\cdot S)\in\cX^1_>$. Conversely, if the process $V=1+H\cdot S$ is such that $V > 0$ and $V_->0$, then
$$
V=1+(V_-V_-^{-1})\cdot V=1+V_-\cdot (V_-^{-1}\cdot (H\cdot S))=1+V_-\cdot ((V_-^{-1}H)\cdot S);
$$
that is, $V=\cE(f\cdot S)$  where $f=V_-^{-1}H$. 

In view of the previous discussion, condition $(iv)$ may be expressed as follows: 

{\it
\begin{enumerate}
	\item[(iv)] for any $\e>0$ there exist $\tilde P\sim P$ with $| \tilde P-P |_{TV} <\e$ and $g\in L(S)$ with $g\Delta S>-1$  such that $\cE(f\cdot S)/\cE(g \cdot S)$ is a local $\tilde P$-martingale for every $f\in L(S)$ with $f \Delta S>-1$.
\end{enumerate}
}

Let now arbitrary $f \in L(S)$ and $g\in L(S)$ with $f\Delta S>-1$ and $g \Delta S>-1$. Straightforward calculations using Yor's product formula $\cE(Z)\cE(\tilde Z)=\cE(Z+\tilde Z+[Z,\tilde Z])$ for semimartingales $Z$ and $\tilde Z$, shows that the ratio $\cE(f\cdot S) / \cE(g\cdot S)$ may be represented  as follows: 
\beq
\label{R0}
\frac{\cE(f\cdot S)}{\cE(g\cdot S)}=\cE( (f-g) \cdot S^g), 
\eeq
where the semimartingale $S^g$, depending only on $S$ and $g$ and \emph{not} on $f$, is given by
\beq
\label{R}
S^g = S -\langle S^c, g\cdot S^c\rangle - \frac{gx}{1+gx}x*\mu = S - (c g) \cdot A - \sum_{s\le .}\frac {g_s\Delta S_s}{1+g_s\Delta S_s}\Delta S_s.
\eeq
\begin{lemm}
\label{superm}
Suppose that $K(|x|^2\wedge |x|)<\infty$, $m$-a.e. Fix
$f \in L(S)$ and $g\in L(S)$ with $f\Delta S>-1$ and $g \Delta S>-1$, and assume that
 \beq \label{eq:integr_techn}
K\Big(\frac{(gx)^2}{1+gx}\Big)<\infty. 
\eeq
Define the predictable process
\beq
\label{Function}
F(f,g):=(f-g)(b-cg)-K\Big((f-g)x\frac{gx}{1+gx}\Big). 
\eeq 
If $F(f,g)\le 0$ $m$-a.e., then $\cE(f\cdot S) / \cE(g\cdot S)$ is a supermartingale. If $F(f,g)= 0$ $m$-a.e., then $\cE(f\cdot S) / \cE(g\cdot S)$ is a local martingale (and a supermartingale).
\end{lemm}

\begin{proof}
We first show that the integral with respect to $K$ appearing in \eqref{Function} is well defined. In view of \eqref{eq:integr_techn}, it suffices to show that $ (fx)(gx) / (1+gx)$ is $K$-integrable. On $\{x\in \R^d:\ gx\ge -1/2\}$,
$$
\frac{|(fx)(gx)|}{1+gx}\le |f||x|I_{\{|x|> 1\}} +|f||g||x|^2I_{\{|x| \le 1\}} \leq |f| (1 + |g|) \big( |x| \wedge |x|^2 \big)
$$
holds, while on $\{x\in \R^d:\ -1<gx\le -1/2,\ -1<fx\le 0\}$ it holds that 
$$
\frac{|(fx)(gx)|}{1+gx}\le \frac{|gx|}{1+gx} \leq 2 \frac{|gx|^2}{1+gx}. 
$$ 
Thus, using \eqref{eq:integr_techn}, on each of the previous two sets the function $(fx)(gx) / (1+gx)$ is $K$-integrable. 
On the other hand, on the set $\{x\in \R^d:\ -1<gx\le -1/2,\ fx>0\}$ the function $ (fx)(gx) / (1+gx)$ is negative. Therefore, the integral with respect to $K$ in \eqref{Function} is well defined, though it could be equal to $-\infty$; however, such a case is excluded under the force of the inequality $F(f,g)\le 0$.

Integrating the $(0,1]$-valued predictable process $\theta:=1/(1+|f|+|g|+K(|x|^2\wedge |x|) - F(f,g))$ with respect to $(f-g) \cdot S^g$ of (\ref{R0}) and (\ref{R}), and rearranging terms we obtain the representation 
$$
\theta (f-g) \cdot S^g = M^\theta + \theta F(f,g) \cdot A,
$$
where $\theta F(f,g) \cdot A$ is a nonincreasing bounded (recall that $A_T \leq 1$) process and 
$$
M^\theta := \theta(f-g)\cdot S^c + \theta\frac{(f-g)x}{1+gx} * (\mu-\nu)\in \cM_{loc}. 
$$
Thus, $(f-g) \cdot S^g$ is a $\sigma$-supermartingale with $(f-g) \Delta S^g >-1$ and, therefore, $(f-g) \cdot S^g$ is a local supermartingale, see \cite{GK}. Since nonnegative local supermartingales and supermartigales, it follows that $\cE(f\cdot S) / \cE(g\cdot S) = \cE((f-g) \cdot S^g)$ is   
a supermartingale.   If $\theta F(f,g)= 0$, then $(f-g) \cdot S^g$ is a local martingale and so is $\cE(f\cdot S) / \cE(g\cdot S)$.
\end{proof}

\section{Existence of the Supermartingale Num\'eraire}
\label{sdeflator}

The aim of this section is to show that condition NA$_1$ implies the existence 
of $g \in L(S)$ with $g \Delta S > -1$ such that $\cE(f \cdot S)/\cE(g \cdot S)$ is a supermartingale for every $f \in L(S)$ with  $f \Delta S > -1$. Under our additional integrability assumption \eqref{ass:sigma-special}, we provide a semi-explicit expression for the integrand $g$  for which  the inequality $F(f,g)\le 0$ of Lemma \ref{superm} holds.

\begin{theo} \label{kara-kard}
Suppose that $K(|x|^2\wedge |x|)<\infty$, $m$-a.e. Under condition $NA_1$, there exists $g \in L(S)$ with $g \Delta S > -1$ such that $\cE(f \cdot S)/\cE(g \cdot S)$ is a supermartingale for every $f \in L(S)$ with  $f \Delta S > -1$. 
\end{theo}

It was shown in \cite{Kara-Kard} that condition NA$_1$ implies the existence of such supermartingale num\'eraire, without any additional assumptions. Since we ask that $K(|x|^2\wedge |x|)<\infty$, Theorem \ref{kara-kard} is a weaker result. However, it is all that we shall need in establishing Theorem \ref{main-I}, and we present in this section for completeness a proof with arguments that are somewhat simpler than the ones in  \cite{Kara-Kard}.

\smallskip

It is instructive to briefly explain how Theorem \ref{kara-kard} will be established in the following subsections. In \S \ref{supnumport} we show that NA$_1$ implies that the cone $I$ of potential relative arbitrage is empty. The fact that $I = \emptyset$ allows one to find solution to a certain concave maximisation problem in \S \ref{concave_max}, whose first order conditions imply the existence of a process $g$ with $g \Delta S > -1$ such that, in the notation of \ref{Function}, $F(f, g) \leq 0$ holds whenever $f \Delta S > -1$. When the L\'evy measures are finite, we argue in \S \ref{discrete} that in fact $F(f, g) = 0$ holds whenever $f \Delta S > -1$. In \S \ref{measur} we show that $g$ can be selected in a predictable way. The only remaining obstacle from constucting the wealth process $\cE (g \cdot S)$ needed in Theorem \ref{kara-kard} is $S$-integrability of $g$, which is established in \S \ref{Integrability}.

\subsection{Local characteristics under absence of immediate arbitrage} 
\label{supnumport}

Lemma \ref{no_imm_arb} that follows uses technical language and notation, but it expresses a simple and intuitive idea. Consider the following property at points $(\omega, t)$: there is a direction $v\in \R^d$ along which the diffusion degenerates, i.e., $c_t(\omega)v=0$, all possible jumps at the next instant are nonnegative and the remaining drift is also nonnegative; furthermore, the resulting investment is nontrivial, in the sense that either a strictly positive jumps is possible or the remaining drift is strictly positive (or both). If  the set of points  $(\omega, t)$ for which the previous property holds is not $m$-null, one can construct an arbitrage opportunity by taking positions in the risky assets proportional to the previous vectors $v$ (which could, of course, change over time and across different  states).

We proceed with formal arguments. Consider two set-valued mappings 
$\Omega \times [0, T] \ni (\omega,t)\mapsto N_{\omega,t}$ and $\Omega \times [0, T] \ni  (\omega,t)\mapsto I_{\omega,t}$, where
$$
N_{\omega,t}:=\{v\in \R^d:\ \ K_{\omega,t}(vx\neq 0)=0,\  c_t(\omega)v=0,\ vb_t(\omega)=0 \}
$$
represents \emph{null investments}, and
$$
I_{\omega,t}:=\{v\in \R^d: \ \ K_{\omega,t}(vx< 0)=0,\  c_t(\omega)v=0,\ vb_t(\omega)\ge K_{\omega,t}(vx)\}\setminus N_{\omega,t}. 
$$ 
The graphs  of the previous mappings are $\bar \cP\otimes \cB(\R^d)$-measurable. 

\begin{lemm} \label{no_imm_arb} Suppose that $K(|x|^2\wedge |x|)<\infty$, $m$-a.e.
Under condition NA$_1$, $m( I \neq \emptyset )=0$ holds. 
\end{lemm}

\begin{proof}
Let  $H$ be a $\bar \cP\otimes \cB(\R^d)$-measurable selector of the set-valued mapping 
$$\Omega \times [0, T] \ni (\omega,t)\mapsto I_{\omega,t}^k:=I_{\omega,t}\cap \{v:\ |v|\le 1,\; K_{\omega,t}(|x|^2\wedge |x|)\le k\}. 
$$
extended by the value $0\in \R^d$ on the set $\{ I^k = \emptyset \}$. 
We claim that $H\cdot S\ge 0$. Indeed, 
 $$
 EHxI_{\{Hx<0\}}*\mu_T=EHxI_{\{Hx<0\}}*\nu_T=EK(HxI_{\{Hx<0\}})\cdot A_T=0
 $$
 and, therefore, for any $n\ge 1$
$$
H\cdot S=HxI_{\{Hx>n^{-1}\}}*\mu + HxI_{\{Hx\le n^{-1}\}}*(\mu-\nu)+
(Hb-K(HxI_{\{Hx>n^{-1}\}}))\cdot A. 
$$
Note that  the first and the third terms in the right-hand are non-negative.  By Doob's inequality, 
\bean
E\sup_{t\le T} \big(HxI_{\{Hx\le n^{-1},\, |x|\le 1\}}*(\mu-\nu)_t\big)^2 
& \le& 4E\big(HxI_{\{Hx\le n^{-1},\, |x|\le 1\}}*(\mu-\nu)_T\big)^2\\
&\le& 4E|Hx|^2I_{\{Hx\le n^{-1},\, |x|\le 1\}}*\nu_T\to 0\\
 &\le& 4EK(|x|^2I_{\{Hx\le n^{-1},\, |x|\le 1\}})\cdot A_T\to 0,
 \eean
as $n\to \infty$ by the dominated convergence theorem. Furthermore,
\bean
E\sup_{t\le T} |HxI_{\{Hx\le n^{-1},\, |x|> 1\}}*(\mu-\nu)_t|
& \le& E|HxI_{\{Hx\le n^{-1},\, |x|> 1\}}*(\mu+\nu)_T\\
& =&2 E|HxI_{\{Hx\le n^{-1},\, |x|> 1\}}*\nu_T\\
& =&2 EK(|xI_{\{Hx\le n^{-1},\, |x|> 1\}})\cdot A_T\to 0. 
\eean
It follows that  $H\cdot S\ge 0$. If $m(I^k \neq \emptyset)>0$, then 
$$
EHx*\mu_T=EHx*\nu_T=EK(HxI_{\{Hx>0\}})\cdot A_T\ge 0
$$ 
and  
$$
E(Hb-K(Hx))\cdot A_T>0 
$$
with at least one of the previous two inequalities being strict. It follows that $H\cdot S\ge 0$, and $P (H\cdot S_T>0) > 0$, in contradiction with condition NA$_1$. Therefore, $m(I^k \neq \emptyset)=0$ holds for every $k\ge 1$, implying that $m(I \neq \emptyset)=0$.
\end{proof}

Note that in the statement of Lemma \ref{no_imm_arb}, a property weaker than NA$_1$ is actually sufficient, which is \emph{absence of immediate arbitrage}: whenever $X = H \cdot S \in \cX^0$, then $X = 0$.

\subsection{Deterministic concave maximisation problem}
\label{concave_max}
Here, for a fixed ``generic" $(\omega,t)$ for which $I_{\omega,t}=\emptyset$, an optimisation problem is considered. The dependence on $(\omega,t)$ is not important, and we can (and will) treat it as a deterministic problem. Its solution gives us a ``candidate'' for the supermartingale num\'eraire. Recall that we are working under the assumption $K (|x| \wedge |x|^2) < \infty$.

Consider the convex set $D:=\{v\in \R^d:  vx> -1, \ K\hbox{-a.e.}\}$, which in general is neither open nor closed, as well as its closure $\bar D=\{v\in \R^d:  vx\ge  -1, \ K\hbox{-a.e.}\}$.
If $K=0$, then $D=\R^d$. Since $y- \log (1+y)\ge 0$ holds for $y\ge -1$, the concave function $\Psi : D \mapsto [-\infty, \infty )$ defined via
\beq
\label{Psi}
\Psi(v):=bv-\frac 12 |c^{1/2} v|^2 - K(vx-\log (1+vx)), \quad v \in D,  
\eeq
is well defined on $D$ though it may take the value $-\infty$. In accordance to \S \ref{supnumport}, we also consider the null-investment linear subspace
$$
N:=\{v\in \R^d:\ K(vx \neq 0) = 0,\  cv=0,\ vb=0 \}
$$
and the set
$$
I:=\{v\in \R^d: \ K(vx < 0) = 0,\ cv=0,\ vb\ge K(vx)\}\setminus N. 
$$
For any $u \in N$, note that $\Psi(v)=\Psi (v + u)=\Psi (v^\perp + u)$ where $v^\perp$, being the projection of $v$ on the subspace $N^\perp$, belongs to $D$. Therefore,   
$$
\Psi^0:=\sup_{v\in D}\Psi(v)=\sup_{v\in D\cap N^\perp}\Psi(v)\ge \Psi (0)=0.
$$ 

Let  $u,v\in D$ be such that $\Psi (u)>-\infty$ and $\Psi(v)>-\infty$. The concave function  
$$
\lambda\mapsto \Psi((1-\lambda )u+\lambda v)= \Psi(u+\lambda (v-u)), \qquad \lambda \in [0,1],
$$ 
is finite and has right derivatives on $[0,1)$. In particular, its right derivative at zero is the directional derivative
$\partial \Psi (u;v-u)$ of $\Psi$ at $u$ in the direction $v-u$.  If one is allowed to differentiate under the sign of the integral  in (\ref{Psi}), then  
$$
\partial  \Psi (u;v-u) = (v-u) (b-cu) -K\Big( (v-u)x \frac{ux}{1+ux}\Big). 
$$
There is a well-known sufficient condition for the validity of such an operation: the derivative of the integrand  with respect to the parameter in a right neighbourhood of zero should be bounded by an integrable  function.  Unfortunately, we cannot use this criterion, and the above identity may fail to be true. However, we claim that the directional derivative $\partial \Psi (u;v-u)$ always dominates the right-hand side of the above formula. Indeed, Jensen's inequality applied to the convex function $\phi(y):=y- \log (1+y)$ gives
$$
\frac {\phi ((1-\lambda )ux+\lambda vx)- \phi (ux)}{\lambda}\le \phi (vx) - \phi (ux)\le \phi (vx). 
$$
Since $\Psi(v)>-\infty$, the function $x\mapsto \phi(vx)$ is $K$-integrable and Fatou's lemma can be applied:   
$$
\limsup_{\lambda\downarrow 0} \frac {K(\phi (ux+\lambda (v-u)x)- \phi (ux))} {\lambda} \le K\Big((v-u)x\frac{ux}{1+ux}\Big). 
$$
Thus, 
\beq
\label{bound}
\partial  \Psi (u;v-u)\ge (v-u)(b-cu)-K\Big((v-u)x\frac{ux}{1+ux}\Big). 
\eeq

\begin{lemm} \label{J-rep}
$I= \{v\in \R^d: \ K[vx < 0 ] = 0,\  \partial \Psi(av;-av)\le 0, \ \forall a>0\}\setminus N$. 
\end{lemm}

\begin{proof}
If  $vx\ge 0$ holds $K$-a.e., and since $K (|x| \wedge |x|^2) < \infty$, one can differentiate under the sign of the integral at point $a v$ in the direction $-a v$ for any $a > 0$ and obtain
$$
- \partial  \Psi(av;-av) = a \Big(v b - a |c^{1/2}v|^2 - a K \Big(\frac{|vx|^2}{1+avx}\Big)\Big), \quad \forall a > 0. 
$$
The inclusion $I \subseteq \{v\in \R^d: \ K[vx < 0 ] = 0,\  \partial \Psi(av;-av)\le 0, \ \forall a>0\}\setminus N$ is straightforward. On the other hand, the expression on the right-hand side of the above identity is greater or equal to zero for arbitrary large $a$ only if $cv=0$ and $v b \ge K(vx)$.
\end{proof}

\begin{prop} 
\label{pointwisemax}
Suppose that $I=\emptyset$. Then there is unique $v^0\in D\cap N^\perp$ such that 
\[
\Psi(v^0)=\sup_{v\in D}\Psi(v)<\infty.  
\]
For any point $v^0\in D$ at which the supremum above is attained
\beq
\label{inequality}
F(v,v^0):=(v-v^0)(b-cv^0)-K\Big((v-v^0)x\frac{v^0x}{1+v^0x}\Big)\le 0, \quad \forall\,  v\in D.
\eeq  
\end{prop}
\begin{proof}
When $K=0$, the claim is obvious. Let $v_n\in D\cap N^\perp$ form a sequence such that $\Psi(v_n)\to \Psi^0$. Without loss of generality, we assume that the function $[0,1] \ni \lambda\mapsto \Psi(\lambda  v_n)$ attains its maximum at $\lambda=1$ (otherwise, replace $v_n$ by $v'_n=\lambda _n v_n$ where $\lambda_n \in [0,1]$ is the point where $[0,1] \ni \lambda\mapsto \Psi(\lambda  v_n)$ attains its maximum). If $(v_n)_n$ contains a  bounded subsequence, we shall work with the latter one, assuming that  it converges to some point $v^0\in  \bar D\cap N^\perp$. In this case, by Fatou's lemma,
\bean
\liminf _n K(v_nx-\log (1+v_nx))&\ge& K(\liminf _n(v_nx-\log (1+v_n x)))\\
&=&K(v^0x-\log (1+v^0x));
\eean
therefore, $\lim_n \Psi(v_n)\le \Psi (v^0)$, i.e., the supremum is attained at $v^0$. Since  $\Psi^0\ge 0$, $v^0 \in D\cap N^\perp$.

It remains to establish that the sequence  $|v_n|$ cannot diverge to infinity. Indeed, if that was the case, we may assume  that the normalised sequence $\bar v_n:=v_n/|v_n|$ converges to some $\bar v \in N^\bot$ with $|\bar v|=1$. Since the inequality  $v_nx\ge -1$ implies that  $\bar v_nx\ge -1/|v_n|$,  it follows that $\bar vx\ge 0$ $K$-a.e. For $a > 0$, let $n_a \geq 1$ be large enough so that $|v_n|\ge 2 a$ holds for all $n \geq n_a$. Since $[0,1] \ni \lambda\mapsto \Psi(\lambda  v_n)$ attains its maximum at $\lambda=1$, for all $a>0$ we have that  $\partial \Psi (a\bar v_n;-a\bar v_n)\le 0$ when $n\ge n_a$.  Hence, by (\ref{bound}) we obtain 
$$
b\bar v_n - a |c^{1/2}\bar v_n|^2 - a K\Big( \frac{|\bar v_n x|^2}{1+a\bar v_nx}\Big)\ge 0, \quad \forall n \geq n_a.
$$
Since $1+a\bar v_nx \geq 1/2$ holds $K$-a.e. for all $n \geq n_a$, the fact that $K (|x| \wedge |x|^2) < \infty$ allows use of the dominated convergence theorem
to imply that
$$
b\bar v - a |c^{1/2}\bar v|^2 - a K \Big(\frac{|\bar v x|^2}{1+a\bar vx}\Big)\ge 0, \quad \forall a > 0,
$$
and we conclude by Lemma \ref{J-rep} that  $\bar v\in I$, contradicting the assumption $I=\emptyset$. 

At any point $v^0\in D$ where $\Psi $ attains its maximum, the inequality $\partial \Psi(v^0,v-v^0)\le 0$ for every $v\in D$ is valid; therefore,  (\ref{inequality}) holds in virtue of (\ref{bound}). 

Uniqueness of $v^0$ on $D\cap N^\perp$ follows because $\Psi$ is strictly concave on the latter set.
\end{proof}

\begin{rem}
\label{Caratheodory}
For $a\in [0,1 )$, let $a\bar D=\{v\in \R^d:\ vx \ge -a\  K\hbox{-a.e.}\}$. In view of the bound  
$$
y-\log (1+y)\le c_a (|y|^2\wedge |y|), \quad y\ge a, 
$$
for a constant $c_a>0$, the dominated convergence theorem implies that $\Psi$ is continuous on $a\bar D$. 
\end{rem}

\subsection{The case of discrete L\'evy measures}
\label{discrete}

Suppose that there exist $x_i\in {\R}^d\setminus \{0\}$ for $i=1,...,N$ such that $K(\{ x_i \}) > 0$ and $K({\R}^d\setminus \{ x_1, \ldots, x_N\}) = 0$. In this case,
$$
D := \{v\in \R^d:\  vx> -1 \ K\hbox{-a.e.}\}=\{v\in {\R}^d:\ vx_i>-1,\ i=1,...,N\}.$$
In particular, $D$ is an open set, which means that $v^0$ that maximises $\Psi$ on an interior point. In the notation of (\ref{inequality}), it follows that $F(v, v^0) = 0$ has to hold for all $v \in D$.

\subsection{Measurability}
\label{measur}
The function $(\omega,t,v)\mapsto K_{\omega,t}(vx\le -1)$ is $\bar \cP\otimes\cB(\R^d)$-measurable and so is the  graph ${\rm Gr}\, D$ 
of the set-valued mapping
$$
(\omega,t)\mapsto D_{\omega,t}=\{v\in \R^d:\ K_{\omega,t}(vx\le -1)=0\}. 
$$ 
Define the function $\Psi : \Omega \times \R_+ \times \R^d \mapsto [- \infty, \infty)$ via
$$
\Psi_{\omega,t}(v):=b_t(\omega)v-\frac 12 |c^{1/2}_t(\omega) v|^2 - K_{\omega,t} (vx-\log (1+vx)).
$$ 

\begin{prop} \label{prop:meas_sel} 
Suppose that $K(|x|^2\wedge |x|)<\infty$, $m$-a.e. Then, the $[0, \infty]$-valued function  
$$
(\omega,t)\mapsto \psi(\omega,t):=\sup_{v\in D_{\omega,t}}\Psi_{\omega,t}(v)
$$
is $\bar \cP$-measurable. If, furthermore, $m( I \neq \emptyset )=0$, then $m (\psi =\infty) = 0$ and there exists a $\cP$-measurable function $g$ such that 
$$
\psi(\omega,t)=\Psi_{\omega,t}(g(\omega,t)),    \quad m\hbox{-a.e.}
$$
\end{prop} 

\begin{proof}
For $a\in [0,1]$, consider the $\bar \cP$-measurable mapping 
$$
(\omega,t)\mapsto a\bar D_{\omega,t}=\{v\in \R^d:\ vx \ge -a\  K_{\omega,t}\hbox{-a.e.}\}
$$ 
with closed (and convex) values; obviously,   $a\bar D_{\omega,t}\subseteq D_{\omega,t}$. For $a\in [0,1 )$, Remark \ref{Caratheodory} implies that $\Psi$ is a Carath\'eodory (measurable in $(\omega,t)$ and continuous in $v$) function on $a\bar D$. Thus, the function 
$$
(\omega,t)\mapsto \psi_a(\omega,t):=\sup_{v\in aD_{\omega,t}}\Psi_{\omega,t}(v)
$$
is $\bar \cP$-measurable---see \cite[Th. 2.27]{Molchanov}, as well as \cite[Chapter 18]{AB} and \cite[ Section 9]{Wagner} for related results.  Let $a_n := 1 - 1/n$. 
The sequence  $\psi_{a_n}$ is increasing to a $\bar \cP$-measurable function which we shall denote $\psi_{1-}$. Fix $(\omega,t)$ and take arbitrary point $v_0\in D_{\omega,t}$.  
The function $a\mapsto ay-\log (1+a y) \in [0, \infty]$ is nondecreasing on $[0,1 )$ for all $y \in \R$; therefore, by the monotone convergence theorem,  
$$
\lim_n K_{\omega,t}(a_nv_0x-\log(1+a_nv_0x))= K_{\omega,t}(v_0x-\log(1+v_0x)). 
$$ 
It follows that   $\psi_{1-} = \psi$; therefore, $\psi$ is $\bar \cP$-measurable. On the $\cP$-measurable set $\{ I = \emptyset \}$, Proposition \ref{pointwisemax} implies that $\psi < \infty$, and that the \emph{non-empty} set-valued  mapping
$$
(\omega,t)\mapsto \{v\in \R^d:\ \psi(\omega,t)=\Psi_{\omega,t}(v)\},
$$
which is $\bar \cP$-measurable, admits  a $\bar \cP$-measurable selector $\bar g$. Therefore, if $m(I \neq \emptyset) = 0$. it suffices to take as $g$ any $\cP$-measurable function $m$-a.e. coinciding with $\bar g$.
\end{proof}

\subsection{Integrability}
\label{Integrability}

Suppose that $K(|x|^2\wedge |x|)<\infty$ and $I = \emptyset$ hold $m$-a.e., and let $g$ be the $\cP$-measurable function of Proposition \ref{prop:meas_sel}. Inequality (\ref{inequality}) applied with $v=0$ and $v^0=g$ gives 
\beq
\label{ineq}
F(0, g) = - g(b-cg) + K\Big(\frac{(gx)^2}{1+gx}\Big) \le 0, \quad m\hbox{-a.e.}  
\eeq
Since $F(0, 0) = 0$ it follows that $F(0, g^n) \leq 0$ holds, where $g^n:=gI_{\{|g|\le n\}}$ for all $n \geq 1$. In virtue of Lemma \ref{superm}, the processes $1 / \cE (g^n\cdot S)$ are supermartingales  for all $n \geq 1$. Recall that condition NA$_1$ is equivalent to the $P$-boundedness of the set $\cX^1_T$ (see Lemma \ref{A4}). Thus, NA$_1$
ensures that the sequence of random variables $(\cE(g^n\cdot S)_T)_n$  is $P$-bounded. In view of Proposition \ref{expbounded},  the sequence  $(g^n\cdot S_T)_n$ is also $P$-bounded.  In Proposition \ref{global}, we shall show that this implies that $g\in L(S)$. We first recall the criterion of the integrability of a vector-valued process with respect to a vector-valued semimartingale---see  \cite{ShCh} and \cite{JS:87}.
 
\begin{prop} A predictable process $g \in L(S)$ if and only if
$$
\big ( |c^{1/2}g|^2+K(|gx|^2\wedge 1)+\big |gb^h-K\big(gxI_{\{|x|\le 1,\;|gx|>1\}}\big)\big |\big)\cdot A_T<\infty.
$$
In the particular case where $g$ is such that $K(gx\le -1)=0$ $m$-a.e. and $K\big(|x|I_{\{|x|>1\}})\big)<\infty$, and recalling that $b =b^h + K\big(xI_{\{|x|>1\}}\big)$, the last relation can be written as
$$
\big ( |c^{1/2}g|^2+K(|gx|^2\wedge 1)+\big |gb-K\big(gx\big (1-I_{\{|x|\le 1,\;|gx|\le 1\}}\big)\big)\big |\big)\cdot A_T<\infty.
$$
\end{prop}

\begin{prop}
\label{global}  
Let $g$ be a predictable process such that \ $K(gx\le -1)=0$ $m$-a.e. and  (\ref{ineq}) holds. Let $g^n:=gI_{\{|g|\le n\}}$, $n\ge 1$, and suppose that the sequence $(g^n\cdot S_T)_n$ of random variables is bounded in probability. Then, $g\in L(S)$. 
\end{prop}

\begin{proof}
Using the decomposition 
$$
S=S_0+S^c+xI_{\{|x|\le 1\}}*(\mu-\nu)+xI_{\{|x|> 1\}}*\mu +b^h \cdot A
$$
and taking into account that  
$$
g^n b^h + K(g^n xI_{\{|x|>1\}})= g^n b \ge |c^{1/2}g^n|^2+K\Big( \frac {(g^nx)^2}{1+g^nx}\Big). 
$$
we obtain that 
$$
g^n\cdot S_T\ge \sum_{j=1}^5 I^n_j
$$ 
where
\bean
 I^n_1&:=&g^n\cdot S^c_T+|c^{1/2}g^n|^2\cdot A_T,\\ 
 I^n_2&:=&g^nxI_{\{|x|\le 1,\;|gx|\le 1\}}*(\mu-\nu)_T+K\big( (g^nx)^2/(1+g^nx)I_{\{|x|\le 1,\;|gx|\le 1\}}\big)\cdot A_T,\\
  I^n_3&:=&g^nxI_{\{|x|\le 1,\;gx> 1\}}*(\mu-\nu)_T+K\big( (g^nx)^2/(1+g^nx)I_{\{|x|\le 1,\;gx> 1\}}\big)\cdot A_T,\\
   I^n_4&:=&g^nxI_{\{|x|>1,\;|gx|\le 1\}}*(\mu-\nu)_T+K\big( (g^nx)^2/(1+g^nx)I_{\{|x|> 1,\;|gx|\le 1\}}\big)\cdot A_T,\\
    I^n_5&:=&g^nxI_{\{|x|> 1,\;gx> 1\}}*\mu_T-K\big( g^nx/(1+g^nx)I_{\{|x|> 1,\;gx> 1\}}\big)\cdot A_T.
\eean 

We treat the five summands above separately.

\noindent {\bf 1.}  In virtue of Lemma \ref{LLN2}
the sequence $I^n_1$ diverges to $+\infty$ a.s. on the set 
$$
\{|c^{1/2}g|^2\cdot A_T=\infty\}
$$ 
and is bounded on its complement. 

\smallskip
\noindent {\bf 2.}  
 Since
 \bean
 K\big( (g^nx)^2/(1+g^nx)I_{\{|x|\le 1,\;|gx|\le 1\}}\big)\cdot A_T&\ge& \frac 12 K\big( (g^nx)^2I_{\{|x|\le 1,\;|gx|\le 1\}}\big)\cdot A_T\\
 &\ge& \frac 12 \big\langle g^nxI_{\{|x|\le 1,\;|gx|\le 1\}}*(\mu-\nu)\big\rangle_T,
 \eean
 we infer, again from Lemma \ref{LLN2} that the sequence $I^n_2$ diverges to  $+\infty$ a.s. 
on the set 
 $$\big\{(gx)^2I_{\{|x|\le 1,\;|gx|\le 1\}}*\nu_T=\infty\big\}$$ and is bounded on its complement. 
 
\smallskip
\noindent {\bf 3.}   Rearranging terms, we rewrite $  I^{n}_{3}$ as follows:
$$
   I^{n}_{3}= I^{n}_{3,1}+  I^{n}_{3,2}
$$
 where  
 \bean 
 I^{n}_{3,1}&:=&g^nxI_{\{|x|\le 1,\;1<gx\le 2\}}*(\mu-\nu)_T+K\big( (g^nx)^2/(1+g^nx)I_{\{|x|\le 1,\;1<gx\le 2\}}\big)\cdot A_T,\\
 I^{n}_{3,2}&:=&g^nxI_{\{|x|\le 1,\;gx> 2\}}*\mu_T-K\big( g^nx/(1+g^nx)I_{\{|x|\le 1,\;gx> 2\}}\big)
   \cdot A_T. 
 \eean
Note that 
\bean 
(1/3)K\big((g^nx)^2I_{\{|x|\le 1,\;1<gx\le 2\}}\big)\cdot A_T&\le&  K\big( (g^nx)^2/(1+g^nx)I_{\{|x|\le 1,\;1<gx\le 2\}}\big)\cdot A_T\\
& \le& (1/2)K\big(g^nx)^2I_{\{|x|\le 1,\;1<gx\le 2\}}\big)\cdot A_T;
\eean
we infer from Lemma  \ref{LLN2} that $(I^{n}_{3,1})_n$   a.s. diverges to $\infty$
on $\big\{K\big(I_{\{|x|\le 1,\;1<gx\le 2\}}\big)\cdot A_T=\infty \big\}$, while it is bounded on its complement. Also 
$$
I^{n}_{3,2}\ge 2I_{\{|g|\le n, \;|x|\le 1,\;gx> 2\}}*\mu_T-K\big(I_{\{|g|\le n, \;|x|\le 1,\;gx> 2\}}\big)
   \cdot A_T
$$
and Lemma  \ref{LLN1} implies that the sequence $I^{n}_{3,2}$ diverges to $+\infty$ a.s. 
on the set 
$$\big\{K\big(I_{\{|x|\le 1,\;gx> 2\}}\big)\cdot A_T=\infty\big\}$$
 and is bounded from below on its complement.

\smallskip
\noindent {\bf 4.}  The sequence $I^n_4$ is  bounded from below and diverges to  $+\infty$ a.s. 
on the set 
$$\big\{K\big( (gx)^2I_{\{|x|> 1,\;|gx|\le 1\}}\big)\cdot A_T=\infty\big\}.$$ 
This follows from the estimates
$$
|g^nxI_{\{|x|>1,\;|gx|\le 1\}}*(\mu-\nu)_T|\le I_{\{|x|>1\}}*(\mu+\nu)_T<\infty 
$$
and  
  $$
 K\big((g^nx)^2/(1+g^nx)I_{\{|x|> 1,\;|gx|\le 1\}}\big)\cdot A_T \ge \frac 12 K\big((g^nx)^2I_{\{|x|> 1,\;|gx|\le 1\}}\big)\cdot A_T.
  $$

\smallskip
\noindent {\bf 5.}  Note that the sequence $|I^n_5|$ is bounded by a finite random variable. Indeed,  
$$
   |I^n_5|\le g^nxI_{\{|x|> 1,\;gx> 1\}}*\mu_T+K\big( g^nx/(1+g^nx)I_{\{|x|> 1,\;gx> 1\}}\big)\cdot A_T
$$  
where the first integral is dominated  by the integral $gxI_{\{|x|> 1,\;gx> 1\}}*\mu_T$ (which is just a finite sum)
while the second is dominated  by $K\big (I_{\{|x|> 1,\;gx> 1\}}\big)\cdot A_T<\infty$. 

\medskip
 
Combining the above five facts, we obtain that  the inequality (\ref{ineq}) implies that 
 $$
|c^{1/2}g|^2\cdot A_T+K\big( (gx)^2\wedge 1)\big)\cdot A_T<\infty.
 $$ 
 Now we check that 
 $$
\big |gb-K\big(gx\big (1-I_{\{|x|\le 1,\;|gx|\le 1\}}\big)\big)\big |\cdot A_T<\infty
 $$
 or, equivalently, that 
 \beq\label{check2}
\big |K\big(gx/(1+gx)- gxI_{\{|x|\le 1,\;|gx|\le 1\}}\big)\big |\cdot A_T<\infty.   
 \eeq
 By the above, 
 $$
 K\big((gx)^2/(1+gx)I_{\{\{|x|\le 1,\;|gx|\le 1\}}\big)\cdot A_T\le \frac 12 K\big((gx)^2I_{\{|gx|\le 1\}}\big)\cdot A_T<\infty
 $$
 and
 $$
 K\big(gx/(1+gx)I_{\{gx> 1\}}\big)\cdot A_T\le K\big (I_{\{gx> 1\}}\big)\cdot A_T<\infty. 
 $$ 
 Also 
 $$
 K\big(gx/(1+gx)I_{\{|x|> 1,\; 0\le gx \le 1\}}\big)\cdot A_T\le K\big(I_{\{|x|> 1\}}\big)\cdot A_T<\infty. 
 $$ 
 Finally, 
 $$
 K\big(|gx|/(1+gx)I_{\{|x|> 1,\; -1\le gx \le 0\}}\big)\cdot A_T<\infty, 
 $$
because in the opposite case the sequence   $I_4^n$ will diverges to infinity on the set of positive probability. These observations show that (\ref{check2}) holds and, therefore, $g \in L(S)$.
\end{proof}

\section{Proof of Implication $(i) \Rightarrow (iv)$ of Theorem \ref{main-I}: the Concluding Step}
\label{num_change}

\subsection{Change of basic securities}

Fix $g\in L(S)$ with $g\Delta S>-1$, and define the semimartingale 
 \beq
 \label{barS}
S^g = S - (cg) \cdot A - \sum_{s\le .}\frac {g_s\Delta S_s}{1+g_s\Delta S_s}\Delta S_s,
\eeq
exactly as in (\ref{R}).
Since $\Delta S^g = \Delta S/(1+g\Delta S)$ we may invert the previous operation and obtain   
 $$
 S=S^g +cg\cdot A+\sum_{s\le .}(g_s\Delta S_s)\Delta S^g_s. 
 $$
\begin{lemm}
With the above notation, $L(S)=L(S^g)$. 
\end{lemm}

\begin{proof}
Let $f\in L(S)$. Then 
$$
|(f,cg)|\cdot A_T\le \frac12|c^{1/2}f|^2\cdot A_T +  \frac12|c^{1/2}g|^2\cdot A_T <\infty,
$$
$$
\sum_{s\le T}\frac {|g_s\Delta S_sf_s\Delta S_s|}{1+g_s\Delta S_s}\le \frac12
 \sum_{s\le T}\frac {|f_s\Delta S_s|^2}{1+g_s\Delta S_s}+\frac12 
 \sum_{s\le T}\frac {|g_s\Delta S_s|^2}{1+g_s\Delta S_s}<\infty. 
 $$
 Thus, $L(S)\subseteq L(S^g)$. To show the opposite inclusion, take $f\in L(S^g)$. The conditions $g\in L(S)$ and $f\in L(S^g)$ implies $f$ and $g$  are integrable with respect to $S^c=(S^g)^c$, i.e. 
 that $|c^{1/2}g|^2\cdot A_T <\infty$ and $|c^{1/2}f|^2\cdot A_T <\infty$. 
 So, as above we have that $|(f,cg)|\cdot A_T$.  Since also  
 $$
 \sum_{s\le t}|(g_s\Delta S_s)(f_S\Delta S^g_s)|\le \frac12
 \sum_{s\le T}|g_s\Delta S_s|^2+\frac12 
 \sum_{s\le T}|f_s\Delta S^g_s|^2<\infty,  
 $$
 we get that $f\in L(S)$, i.e. the inclusion $L(S^g)\subseteq L(S)$ holds. 
\end{proof}

\begin{lemm} 
\label{lshift}
With the previous notation, we have the identity 
\beq
\label{shift}
\{h\in L(S^g):\ h\Delta S^g >-1\}=\{f\in L(S):\ f\Delta S>-1\}-g.
\eeq
Furthermore, for $f\in L(S^g)$ with $f\Delta S>-1$, 
\beq
\label{barS_2}
\frac{\cE(f\cdot S)}{\cE(g\cdot S)}=
\cE((f-g)\cdot S^g). 
\eeq
 \end{lemm}

\begin{proof}
If $h:=f-g$ belongs to the set in the right-hand side of (\ref{shift}), then $h\in L(S)=L(S^g)$ and  
$$
h\Delta S^g= (f-g)\Delta S^g = \frac{(f-g)\Delta S}{1+g\Delta S}=\frac{1+f\Delta S}{1+g\Delta S}-1>-1.
$$
That is, $f-g$ belongs to the set in the left-hand side of (\ref{shift}).  On the other hand, let $h$ belong to the left-hand side of (\ref{shift}). Then $f:=h+g$ belongs to $L(S)$. Substituting the expression $\Delta S= \Delta S^g / (1- \Delta S^g)$  we obtain
$$
f\Delta S= (h+g)\Delta S^g = \frac{(h+g)\Delta S^g}{1-g\Delta S^g}=\frac{(1+h\Delta S^g}{1-g\Delta S^g}-1>-1.
$$
Therefore, $h$ belongs to the right-hand side of (\ref{shift}). Formula (\ref{barS_2}) is (\ref{R0}).  
\end{proof}

\begin{rem}
The statement of Lemma \ref{lshift} implies that   
\beq
\label{idset1}
1+\cX^1_>(S^g)=\cE^{-1}(g\cdot S)(1+\cX^1_>(S)). 
\eeq
It is easy to deduce from the above equality of the two sets that  
\beq
\label{idset2}
1+\cX^1(S^g)=\cE^{-1}(g\cdot S)(1+\cX^1(S)). \\
\eeq
Indeed, let  
$1+H\cdot S^g \ge 0$. Then $1+(1/2)H\cdot S^g>0$ and, in virtue of (\ref{idset1}), there is $\tilde H$ in $L(S)$ such that $1+\tilde H\cdot S>0$ and $\cE(g\cdot S)(2+H\cdot S^g)=
2(1+\tilde H\cdot S)$.  It follows that  
$$
\cE(g\cdot S)(1+H\cdot S^g)=2(1+\tilde H\cdot S)-\cE(g\cdot S)=1+(2\tilde H-\cE_-(g\cdot S))\cdot S\in 1+\cX^1(S). 
$$
Thus, we have the inclusion ``$\subseteq$'' in (\ref{idset2}). Since $\cE(-g\cdot S^g)=\cE^{-1}(g\cdot S)$, the same arguments work in the proof of the reverse inclusion.  
\end{rem}

\subsection{Supermartingale num\'eraire is local martingale num\'eraire under equivalent probability}
The following result is the last ingredient needed to complete the proof of implication $(i)\Rightarrow (iv)$ of Theorem \ref{main-I}.
\begin{lemm} \label{lem:final}
Let $g \in L(S)$ with $g \Delta S > -1$ be such that $\cE(f\cdot S) / \cE(g\cdot S)$ is a $P$-supermartingale for all $f \in L(S)$ with $f \Delta S > -1$. Then, for every $\e>0$ there exists a probability $\tilde P\sim P$ with $|P-\tilde P|_{TV}<\e$ such that $\cE(f\cdot S) / \cE(g\cdot S)$ is a local $\tilde P$-martingale for all $f \in L(S)$ with $f \Delta S > -1$. 
\end{lemm}  
\begin{proof} 
Let $S^g$ be given by (\ref{barS}). It follows from Lemma \ref{lshift} that for every $h\in L(S)$ for which $h\Delta S>-1$ the process $\cE(h\cdot S^g)$ is a supermartingale, which implies that $E\cE(h\cdot S^g)_T\le 1$. In other words, $EH\cdot S_T\le 0$ for every $H\cdot S^g \in \cX^1_>$, hence, for every 
integral $H\cdot S^g$ provided that it is bounded from below. This means that the probability $P$ is a separating measure for $S^g$, and an application 
of Theorem \ref{sig}  implies the existence of   $\tilde P\sim P$ with $|P-\tilde P|_{TV} <\e$ such that $S^g$ is a $\sigma$-martingale with respect to $\tilde P$. 
It follows that all bounded from below integrals $H\cdot S^g$ are local $\tilde P$-martingales. In other words, $\cE(f\cdot S)/\cE(g\cdot S)$ is a local $\tilde P$-martingale for all $f\in L(S)$ with $f\Delta S>-1$, which is exactly what was required.
\end{proof} 

\subsection{Completing the proof of Theorem \ref{main-I}}
\label{subsec:proof_final}
We now discuss how the combination of previous results imply Theorem \ref{main-I}. Only implication $(i) \Rightarrow (iv)$ and the statement regarding L\'evy measures with finite support needs argument. Fix $\e > 0$. A combination of Lemma \ref{special} and Theorem \ref{kara-kard} imply that there exists $P' \sim P$ with $|P' - P|_{TV} \leq \e / 2$ and $g \in L(S)$ with $g \Delta S > -1$ such that $\cE (f \cdot S) / \cE (g \cdot S)$ is a $P'$-supermartingale for all $f \in L(S)$ with $f \Delta S > -1$. Then, Lemma \ref{lem:final} implies that there exists $\tilde P \sim P'$ with $|\tilde P - P|_{TV} \leq \e / 2$ (and, therefore, $|\tilde P - P|_{TV} \leq \e$) such that $\cE (f \cdot S) / \cE (g \cdot S)$ is a local $\tilde P$-martingale for all $f \in L(S)$ with $f \Delta S > -1$, which establishes implication $(i) \Rightarrow (iv)$. When the L\'evy measure in $m$-a.s. concentrated in a finite number of points, Theorem \ref{kara-kard} can be used directly and we can choose $P' = P$. According to \S \ref{discrete}, $\cE (f \cdot S) / \cE (g \cdot S)$ is already a local $P$-martingale for all $f \in L(S)$ with $f \Delta S > -1$, which means we can select $\tilde P = P$.

\begin{rem}
For the proof of implication $(i) \Rightarrow (iv)$ of Theorem \ref{main-I}, we first perform a probability measure change with Lemma \ref{special} and then use Theorem \ref{kara-kard} in order to establish the existence of supermartingale num\'eraire. Making a direct reference to results of \cite{Kara-Kard} or \cite{Kard13} one can obtain a more precise result: condition NA$_1$ implies that the supermartingale num\'eraire $\cE(g \cdot S)$ under $P$ exists, and, additionally, in any neighbourhood (in total variation) of $P$ there exists a probability $\tilde P \sim P$ under which the \emph{same} wealth process $\cE(g \cdot S)$ is actually a local martingale num\'eraire. 
\end{rem}

\appendix

\section{Arbitrage Theory Revisited}
\label{sec:arbitrage}

\subsection{Condition NA$_1$: equivalent formulations}

We discuss equivalent forms of condition NA$_1$ in the context of a general abstract setting, where the model is given by specifying the wealth processes. The advantage of this  generalisation is that one may use only elementary properties without any reference to stochastic calculus and integration theory.

Let  $\cX^1$ be a convex set  of c\`adl\`ag processes $X\ge -1$ with $X_0=0$, containing the zero process. For $x\ge 0$ we put  $\cX^x=x\cX^1$, and note that $\cX^x \subseteq \cX^1$ when $x\in [0,1]$. 
Set $\cX:={\rm cone}\,\cX^1={\R}_+\cX^1$ and define the sets of terminal random variables 
$\cX^1_T:=\{X_T:\  X\in \cX^1\}$ and   $\cX_T:=\{X_T:\ X\in \cX\ \}$. In this setting, the elements of $\cX$ are interpreted as  admissible wealth processes  starting from zero initial capital; the  elements of $\cX^x$ are  called $x$-admissible.

\begin{rem}[``Standard'' model]
In the typical example, a $d$-dimensional semimartingale $S$ is given and $\cX^1$ is the set of stochastic integrals $H\cdot S$ where $H$ is $S$-integrable and $H\cdot S\ge -1$. Though our main
result deals with the standard model,  the discussion of  basic definitions and their relations with  concepts of the arbitrage theory is more natural in more general framework.
\end{rem}

Define the set of strictly $1$-admissible processes
$\cX^1_> \subseteq \cX^1$ composed of  $X\in \cX^1$
such that $X > -1$ and $X_->-1$. The sets $x+\cX^x$, $x+\cX^x_>$ etc., $x\in {\R}_+$, have  obvious interpretation. 
We are particularly interested in the set  $1+\cX^1_>$. Its elements are strictly positive wealth processes starting with unit initial capital, and may be thought as \emph{tradeable num\'eraires}.

For $\xi\in L^0_+$, define the \emph{superreplication price} $\bar x (\xi) := \inf \{x:\ \xi\in x+\cX^x_T-L^0_+\}$. We say that the wealth-process family $\cX$ satisfies condition NA$_1$ (\emph{No Arbitrage of the 1st Kind}) if $\bar x(\xi)>0$ holds for every $\xi\in L^1_+\setminus\{0\}$. Alternatively, condition NA$_1$ can be defined via
$$
\Big(\bigcap_{x>0}\{x+\cX^x_T-L^0_+\}\Big)\cap L^0_+=\{0\}. 
$$

The family $\cX$ is said to satisfy condition NAA$_1$ (\emph{No Asymptotic Arbitrage of the 1st Kind}) if for any sequence $(x^n)_n$ of positive numbers with $x^n\downarrow 0$ and any sequence of value processes $X^n   \in \cX$ such that $x^n+X^n\ge 0$, it holds that $\limsup_n P(x^n+X^n_T\ge 1)=0$.

Finally, the family $\cX$ satisfies condition NUPBR (\emph{No Unbounded Profit with Bounded Risk}) if the set $\{X_T:\ X\in \cX^1_>\}$ is $P$-bounded. Since $(1/2)\cX^1_T=\cX^{1/2}_T\subseteq \{X_T:\ X\in \cX^1_>\}$, the sets  $\{X_T:\ X\in \cX^1_>\}$ and $\cX^1_T$ are $P$-bounded simultaneously. 

The next result shows that all three previous notions coincide.

\begin{lemm} 
\label{A4}
NAA$_1$ $\ \Leftrightarrow\  $ NUPBR $\ \Leftrightarrow\ $ NA$_1$.
\end{lemm} 

\begin{proof}
NAA$_1 \ \Rightarrow\  $NUPBR: If $\{X_T:\ X\in \cX^1_>\}$ is not $P$-bounded,   $P(1+\tilde X^n_T\ge n)\ge \e>0$ holds for a sequence of $\tilde X^n\in \cX^1_>$, and we obtain a violation of NAA$_1$ with  $n^{-1}+n^{-1}\tilde X^n_T$.

\smallskip

\noindent NUPBR$\ \Rightarrow\  $ NA$_1$: If NA$_1$ fails,  there exist $\xi\in L^0_+\setminus \{0\}$ and a sequence $X^n\in \cX^{1/n}$ such that $1/n+X^n\ge \xi$. Then, the sequence $nX^n_T\in \cX^1$ fails to be $P$- bounded, in violation of NUPBR. 

\smallskip

\noindent NA$_1\ \Rightarrow\  $ NAA$_1$: If the implication fails, then there are sequences $x^n\downarrow 0$ and 
$X^n\ge -x^n$ such that $P(x^n+X^n_T\ge 1)\ge 2\e>0$. By the von Weizs\"acker theorem (see \cite{von W} or \cite[5.2.3]{KSaf}), any sequence of random variables bounded from below contains a subsequence converging in Cesaro sense a.s. as well as its all further subsequences. We may assume without loss of generality that already for $\xi^n:=x^n+X^n_T$ the sequence $\bar \xi^n:=(1/n)\sum_{i=1}^n\xi_i$ converges $\xi \in L^0_+$.   Note that  $\xi\neq 0$. Indeed, 
\bean
\e(1-P(\bar \xi^n\ge \e)) &\ge&       \frac 1n\sum_{i=1}^nE\xi^iI_{\{\bar \xi^n<\e\}}\ge \frac 1n\sum_{i=1}^nE\xi^iI_{\{\xi^i\ge 1,\;\bar \xi^n<\e\}}\\
&\ge& \frac 1n\sum_{i=1}^nP(\xi^i\ge 1,\;\bar \xi^n<\e)\ge \frac 1n\sum_{i=1}^n(P(\xi^i\ge 1)-P(\bar \xi^n\ge \e))\\
&\ge &2\e-P(\bar \xi^n\ge \e). 
\eean
It follows that $P(\bar \xi^n\ge \e))\ge \e/(1-\e)$. Thus, 
$$
E(\xi\wedge 1)=\lim_n E(\bar\xi^n\wedge 1)\ge \e^2/(1-\e)>0.  
$$ 
It follows that there exists $a>0$ such that $P(\xi\ge 2a)>0$. 
In view of  Egorov's theorem, one can find a measurable  set $\Gamma \subseteq \{\xi \ge a\}$ with $P(\Gamma)>0$ on which $x^n+X^n\ge a$ holds for all sufficiently large $n$. But this means that the random variable $aI_{\Gamma}\neq 0$  can be super-replicated starting with arbitrary small initial capital, in contradiction with the assumed condition NA$_1$.
\end{proof}

\begin{rem}[On terminology and bibliography] 
Conditions NAA$_1$ and NA$_1$ have  clear financial meanings, while $P$-boundedness of  the set  $\cX^1_T$ at first glance looks as a technical condition---see \cite{DS}.
The concept of NAA$_1$ first appeared in \cite{Kab-Kram:94} in a much more general context of large financial markets, along with another fundamental  notion NAA$_2$ (No Asymptotic Arbitrage of the 2nd Kind). The $P$-boundedness of $\cX^1_T$ was discussed in  \cite{K97} (as the BK-property), in the framework of a model given by value processes; however, it was overlooked that it actually coincides with NAA$_1$ for the ``stationary'' model. This condition appeared under the acronym NUPBR  in \cite{Kara-Kard}, and was shown  to be equivalent to NA$_1$ in \cite{Kard10}.  
\end{rem}

\subsection{NA$_1$ in terms of conditions NA and NFLVR}

Remaining in the framework of the abstract model of the previous subsection, we provide here results on the relation of condition NA$_1$ with other fundamental notions of the arbitrage theory, cf. with \cite{K97}.  

Define the convex sets  
$C:=(\cX_T-L^0_+)\cap L^{\infty}$ and
denote by $\bar C$, $\tilde C^*$, and $\bar C^*$ respectively
the norm closure, the  sequential  weak$^*$ closure, and weak$^*$ closure of $C$ in $L^{\infty}$. Conditions NA, NFLVR, NFLBR, and NFL are respectively  defined via
$$
C\cap L^{\infty}_+=\{0\},\qquad 
\bar C\cap L^{\infty}_+=\{0\}, \qquad\tilde C^{*}\cap L^{\infty}_+=\{0\},
\qquad\bar C^*\cap L^{\infty}_+=\{0\}.
$$
Consecutive inclusions induce the hierarchy of these
properties:
$$
\begin{array}{ccccccc}
C& \subseteq &\bar C&\subseteq & \tilde C^{*}&\subseteq& \bar C^{*} \\
NA &\Leftarrow & NFLVR & \Leftarrow &
NFLBR & \Leftarrow & NFL.
\end{array}
$$

\begin{lemm}  NFLVR $\ \ \Rightarrow \ \ $ NA $\ \&\ $ NA$_1$. 
\end{lemm}
\begin{proof}
Assume NFLVR. Condition NA follows trivially. If $NA_1$ fails, 
then there exists $[0,1]$-valued $\xi \in L^0_+ \setminus \{0 \}$ such that for each $n\ge 1$ one can find $X^n\in\cX^{1/n}$ with $1/n+X^n_T\ge \xi$. Then 
the random variables $X^n_T\wedge \xi$ belong to $C$ and converge uniformly to $\xi$, contradicting NFLVR. 
\end{proof}

For the sequel we need some further assumptions. We shall call a model \emph{natural} if the elements of $\cX$ are adapted processes and for any $X\in \cX$, $s\in [0,T[$, and $\Gamma\in \cF_s$ the process $\tilde X:=I_{\Gamma\cap \{X_s\le 0\}}I_{[s,T]}(X-X_s)$ is an element of $\cX$. In words, a model is natural if an investor deciding to start trading  at time $s$ when the event $\Gamma$ happened, can use from this time, if $X_s\le 0$,  the investment strategy that leads to the value process with the same increments as $X$.

\begin{lemm} \label{lem:NA_natural}
Suppose that the model is natural. Let $X\in \cX$. If NA holds, then $X\ge - \|X_T^-\|_{\infty}$.
\end{lemm}
\begin{proof}
Let $\lambda := \|X_T^-\|_{\infty}$. If $P(X_s<-\lambda)>0$, then the process $\tilde X:=I_{\{X_s<-\lambda\}}I_{[s,T]}(X-X_s)$ belongs to $\cX$, the random variable $\tilde X_T\ge 0$ and $P(\tilde X_T>0)>0$ in violation of NA.
\end{proof}

\begin{prop}  
Suppose that the model is natural, and, additionally, for every $n\ge 1$ and $X\in \cX$ with $x\ge -n^{-1}$ the process $n X\in \cX^1$. Then, NFLVR $\ \ \Leftrightarrow \ \ $ NA $\ \&\ $ NA$_1$.
\end{prop}
\begin{proof}
We only have to show implication $\Leftarrow$.
If NFLVR fails,
there are  $\xi_n\in C$ and $\xi\in L^\infty_+\setminus\{0\}$ such that $\|\xi_n-\xi \|_\infty \le n^{-1}$.
By definition, $\xi _n\le
\eta _n=X^n_T$ where
$X^n\in \cX$. Obviously, $\|\eta _n^-\|_{\infty}\le n^{-1}$ and,
since NA holds, $nX^n\in \cX^1$ in virtue of Lemma \ref{lem:NA_natural} and our hypothesis. By the von Weizs\"acker theorem, we may
assume that $\eta_n\to \eta$ a.s.
Since $P(\eta>0)>0$, the sequence $nX^n_T\in \cX^1_T$ tends
to infinity with strictly positive probability,
violating condition NUPBR, or, equivalently, NA$_1$.
\end{proof}

Examples showing that conditions NFLVR, NA, and NA$_1$ are all different can be found in  \cite{Her-Her}.

\smallskip

Assume now that $\cX^1$ is a subset of the space of semimartingales $\cS$, equipped with the Emery topology given by the quasinorm 
$$
{\bf D}(X):=\sup \{E1\wedge|H\cdot X_T|:\
H\ \hbox{is predictable, }\ |H|\le 1\}.
$$

Define the condition ESM as the existence of
$\tilde P\sim P$ such that $\tilde EX_T\le 0$ for all processes $X\in \cX$.  A probability $\tilde P$ with such property is referred to as \emph{equivalent separating measure}. According to the Kreps--Yan separation theorem, conditions NFL and ESM are equivalent. The next result under the specific assumptions is given in \cite{K97}, heavily using previous work, especially of \cite{DS, DS:98}.

\begin{theo}
\label{FTAP1} Suppose that $\cX^1$ is closed in $\cS$, and that the following concatenation property holds: for any $X,X'\in \cX^1$ and any bounded predictable processes $H,G\ge 0$ such that
$HG=0$  the process  $\tilde X:=H\cdot X+G\cdot X'$ belongs to $\cX^1$ if it satisfies the inequality $\tilde X\ge -1$.

Then, under condition NFLVR it holds that $C=\bar C^{*}$ and, as a corollary, we have
$$
NFLVR \Leftrightarrow NFLBR \Leftrightarrow NFL  \Leftrightarrow ESM.
$$
\end{theo}

\subsection{Deflators}

In accordance to the analogous definition in the main text for the ``standard'' model, a process $Z>0$ is a strictly positive \emph{supermartingale deflator} (resp., \emph{local martingale deflator}) if $Z(1+X)$ is a supermartingale (resp., a local martingale) for each $X\in \cX^1$. Since Fatou's lemma implies that local martingales bounded from below are supermartingales, local martingale deflators are automatically supermartingale deflators. 

Existence of  a supermartingale deflator $Z$ implies NA$_1$. Indeed, $EZ_T(1+X_T)\le 1$ holds for every $X\in \cX^1$.  Since boundedness in $L^1$ implies boundedness in probability (due to Chebyshev's inequality), the set  $\{Z_T(1+X_T): \ X\in \cX^1\}$ is $P$-bounded, and so is the set $\{1+X_T: \ X\in \cX^1\}$.   

Again, in accordance to the definition for the ``standard'' model,
a wealth process $V \in 1+\cX^1_>$ is a \emph{supermartingale num\'eraire} (resp., \emph{local martingale num\'eraire})  if $1/V$ is a strictly positive supermartingale deflator (resp., local martingale deflator). The supermartingale num\'eraire (as well as the local martingale num\'eraire) if it exists,  is unique. Indeed, the only case where both  a strictly positive processes starting from unit initial value and its reciprocal are supermartingales is when this process is identically equal to one.  To see this, note that the  function $x\mapsto x^{-1}$ is strictly convex and decreasing on $(0,\infty)$; therefore, for a supermartingale $R>0$ with $R_0=1$ it holds that $E R^{-1}_t \ge (E R_t)^{-1} \ge 1$, where the first inequality is equality only when the random variable  $R^{-1}_t$ is  equal to a constant (a.s.). Since $R^{-1}$ is also supermartingale, then $ER^{-1}_t\le 1$; this is consistent with the above only if $R \equiv 1$, a.s. 

\smallskip

The following result, which is a particular case of \cite[Theorem 1.7]{Kard13}, provides a criterion relating  NA$_1$ and the existence of a supermartingale deflator. 
\begin{theo}
Suppose that $1+\cX^1$ is fork-convex, i.e., for 
every $s\in [0,T]$, an $\cF_s$-measurable $[0,1]$-valued random variable $\alpha_s$, $X\in 1+\cX^1$ and $X' \in 1+\cX^1_{>}$, $X''\in 1+\cX^1_{>}$, the process
$$
XI_{[0,s)}+\big (\alpha_s (X_s/X'_s)X'_t  + (1-\alpha_s) (X_s/X''_s)X''_t\big)
I_{[s,T]}
$$ 
belongs to $1+\cX^1$. Then, condition NA$_1$ is equivalent to existence of the supermartingale num\'eraire. 
\end{theo}

\subsection{Standard model}
For the ``standard'' model where wealth processes are of the form $X=H\cdot S$, where $S$ is a $d$-dimensional semimartingale and $H$ runs through the space of predictable processes $L(S)$ for which the stochastic integral is defined,  closedness in $\cS$ of $\cX^1$ follows from M\'emin's theorem in \cite{Memin}; therefore, Theorem \ref{FTAP1} applies. If $S$ is bounded (resp., locally bounded), it is straightforward to check that under a separating measure it is a martingale (resp., local martingale). Without any local boundedness assumption on $S$, we have the following result from \cite{DS:98}, a short proof of which is given in Appendix \ref{existsigma}:  
\emph{In any neighborhood (in total variation) of a separating measure 
there exists an equivalent probability measure under which $S$ is a $\sigma$-martingale.} It follows that, if NFLVR holds, the process $S$ is a $\sigma$-martingale with respect to some probability measure $P'\sim P$ with density process $Z'$. Therefore, for any process $X \equiv H\cdot S$ from $\cX^1$, the process $1+X$ is a local martingale with respect to $P'$, or equivalently,  $Z'(1+X)$ is a local martingale with respect to $P$; therefore, $Z'$ is a local martingale deflator.


\section{Existence of Equivalent $\sigma$-Martingale Measures}
\label{existsigma}

Let $S$ be a $d$-dimensional semimartingale written in canonical decomposition
$$
S=S_0+S^c+xI{\{|x|\le 1\}}*(\mu-\nu)+xI_{\{|x|> 1\}}*\mu +B^h,
$$
with triplet of predictable characteristics $B^h=b^h\cdot A$, $\langle S^c\rangle=c\cdot A$, $\nu(dt,dx)=dA_tK_t(dx)$ with $A_T\le 1$ (see Subsection \ref{subsec:canon_decomp} for details). Let 
$$\cX^1:=\{H\cdot S:\ H\cdot S\ge  -1,\ H\in L(S) \}.$$
A probability measure $P$ is called {\it separating measure}, if $E X_T \le 0$ for all $X \in \cX^1$. In particular, if all processes in $\cX^1$ are supermartingales with respect to $P$, then $P$ is a separating measure. 

Theorem  \ref{sig} below is due to Delbaen and Schachermayer \cite{DS:98}. The proof given here is borrowed, with some simplifications,  from \cite{K97}.  The argument  in Lemma \ref{Girsanov} is the same as in \cite{KL}; it allows us to avoid references to Hellinger processes used in \cite{K97}.  

\begin{theo}
\label{sig}
Suppose that $P$ is a separating measure. Then for any $\e>0$ there exists $\tilde P\sim P$ such that $|P-\tilde P|_{TV} \le \e$ and  $S$ is a $\sigma$-martingale with respect to $\tilde P$.  
\end{theo}

The rest of the section is concerned with the proof of Theorem \ref{sig}. We start with the construction of the density process of bounded variation defining $\tilde P$. 
\begin{lemm}
\label{Girsanov}
 Let $\e>0$ and let $Y\colon\Omega\times [0,T]\times \R^d\to (0, \infty)$ be a $\tilde \cP$-measurable function such that the following conditions are satisfied $m$-a.e.: 
\begin{enumerate}
	\item[(a)]  $K_t(|Y-1|)\le \e/2$;
	\item[(b)] $I_{\{\Delta A>0\}}K_t(Y-1)=0$. 
\end{enumerate}
Then, the process  $Z:=\cE((Y-1)*(\mu-\nu))$ is a strictly positive uniformly integrable martingale,  $\tilde P:=Z_TP$ is a probability measure and $|\tilde P-P|_{TV} \le \e$. Furthermore, the  triplet $(\tilde B,\tilde C,\tilde Y)$ of predictable characteristics of $S$ with respect to $\tilde P$ has the form   
 $$\tilde B^h=B^h+(Y-1)xI_{\{|x|\le 1\}}*\nu, \qquad \tilde C=C, \qquad \tilde \nu=Y\nu.
 $$
\end{lemm}

\begin{proof}
Note that 
$$
|Y-1|*\nu_T=\int_{[0,T]} K_t(|Y-1|)dA_t\le (\e/2)A_T\le  \e/2,
$$
in view of $(a)$. The process  $M:=(Y-1)*(\mu-\nu)$ is a    martingale. In virtue of $(b)$,
$$
\Delta M_t=\int (Y(t,x)-1)\mu(\{t\},dx)-K_t(Y-1)\Delta A_t >-1. 
$$
Thus, $Z=\cE(M)$ is a strictly positive local martingale of bounded variation satisfying the linear equation  
$Z=1+Z_-\cdot M$.  Since 
\bean
E\sup_{t\le T} |Z_t-1|&=&E\sup_{t\le T} |Z_-(Y-1)*(\mu-\nu)_t|\le EZ_-|Y-1|*(\mu+\nu)_T\\
&=&2EZ_-|Y-1|*\nu_T=2EZ_T|Y-1|*\nu_T\le \e,  
\eean
the process $Z$ is uniformly integrable martingale and $|\tilde P-P|_{TV} = E|Z_T-1|\le \e$. 
The form of the triplet of predictable characteristics of $S$ follows from Girsanov's theorem---see \cite{JS:87}.
\end{proof}

\begin{rem} 
\label{rem:sigma}
In the notation of Lemma \ref{Girsanov}, the semimartingale $S$ is a $\sigma$-martingale under $\tilde P$ if and only if $\tilde K_t(|x|\wedge |x|^2)<\infty$ and $\tilde b^h_t + \tilde K_t\big (xI_{\{|x|> 1\}}\big )=0$ holds $m\hbox{-a.e.}$, which translates to
$$
K_t(Y (|x|\wedge |x|^2) )<\infty,\qquad  b^h_t+K_t\big (Y x - xI_{\{|x|\le 1} \big ) = 0, \qquad m\hbox{-a.e.}
$$
The above criterion for the $\sigma$-martingale property appeared in \cite{K97}; see also \cite{Kallsen}, \cite{JS:87}. 
\end{rem}

Continuing,   denote by $\bar\R^d$ the one-point compactification of $\R^d$. Let   $C(\bar\R^d)$  be the compact space of  continuous functions on $\bar\R^d$ equipped by  the uniform norm and the Borel $\sigma$-algebra $\cB(C(\bar\R^d))$; furthermore,  let ${\bf Y}={\bf Y}(\R^d)$ be the subset of this space formed by the strictly positive continuous functions.  We define, for every $(\omega,t)$,  the convex  sets   
\bean
\Gamma^\e_{\omega,t}&:=&\big \{Y\in {\bf Y}: \  K_t ((|x|\wedge |x|^2) Y)<\infty, \  K_t(|Y-1|)\le \e/2, \
  I_{\{\Delta A>0\}}K_t(Y-1)=0\big \},\\
\Gamma_{\omega,t}&:=&\big \{Y\in{\bf Y}: \  K_t\big (|xY-xI_{\{|x|\le 1\}}|\big )<\infty, \  b^h_t+K_t\big (xY-xI_{\{|x|\le 1\}}\big )=0\big \}. 
\eean
The graphs  of the set-valued mappings $(\omega,t)\mapsto\Gamma^\e_{\omega,t}$ and $(\omega,t)\mapsto\Gamma_{\omega,t}$ are  $\cP\otimes \cB(C(\bar\R^d))$-measurable sets;  indeed, they are intersections of level sets of functions which are $\cP$-measurable in $(\omega,t)$ and continuous in $Y$, therefore,   $\cP\otimes \cB(C(\bar\R^d))$-measurable. 

\smallskip
The crucial result of the proof of Theorem  \ref{sig} is the following.
\begin{lemm} 
\label{crucial}
$m(\Gamma^\e \cap \Gamma \neq \emptyset )=0$. 
\end{lemm}

Assuming for the moment the validity of Lemma \ref{crucial}, we explain how Theorem \ref{sig} follows. Applying the measurable selection theorem to the set-valued mapping  $(\omega,t)\mapsto\Gamma^\e_{\omega,t}\cap \Gamma_{\omega,t}$ we obtain a $\cP$-measurable $C(\bar\R^d)$-valued function $(\omega,t)\mapsto Y({\omega,t, \cdot})$ such that $Y({\omega,t, \cdot})\in \Gamma^\e_{\omega,t}\cap \Gamma_{\omega,t}$, $m$-a.e. Note that the mapping $(\omega,t,x)\mapsto Y({\omega,t,x})$ from $\Omega\times [0,T]\times \R^d$ into $(0, \infty)$ is measurable with respect to the $\sigma$-algebra $\tilde \cP=\cP\otimes \cB^d$ (due to continuity in $x$).  Using  Lemma \ref{Girsanov}, we define the new probability measure $\tilde P \sim P$ under which the local characteristics of $S$ are as follows: 
$$
\tilde b^h_t=b^h_t+K_t\big ((Y-1)xI_{\{|x|\le 1}\big ), \qquad\tilde c=c, \qquad \tilde K_t(dx)=Y(t,x)K_t(dx).
$$  
Using Remark  \ref{rem:sigma}, we obtain that $S$ is a $\sigma$-martingale with respect to $\tilde P$, which concludes the proof of Theorem \ref{sig}. Therefore, it just remains to provide the

\begin{proof}[Proof of Lemma \ref{crucial}]
We first consider the case  $d=1$.  Fix  $(\omega,t) \in \Omega \times [0, T]$. On $\Gamma^\e_{\omega,t}$ it holds that $K_t\big (|xY-xI_{\{|x|\le 1\}}|\big )<\infty$ and the affine mapping $\chi_{\omega,t}: Y\mapsto K_{\omega,t}  (xY-xI_{\{|x|\le 1\}})$ is well-defined, with the image $\chi_{\omega,t}(\Gamma^\e_{\omega,t})$ being a convex set, hence, an interval. Therefore, in the considered scalar case we need to check that $-b^h_t$ belongs to the previous interval, with the possible exception of a $m$-null set. Define the predictable process $r:=\sup\{x:\ K( (-\infty,x))=0\}$, as well as $R:=\inf \{x:\  K( (x,\infty))=0\}$. Note that 
\bean
EI_{\{r>-n\}}x^-I_{\{x\le  -n\}}*\mu_T&=&
EI_{\{r>-n\}}x^-I_{\{x\le  -n\}}*\nu_T\\
&=&\int_{[0,T]}I_{\{r>-n\}}K_t(x^-I_{\{x\le  -n\}})dA_t=0. 
\eean
Thus, the finite increasing process $I_{\{r>-n\}}x^-I_{\{x< -1\}}*\mu$ is locally bounded as having jumps do not exceeding $n$. 
The  processes  $I_{\{r>-n\}}\cdot S^c$, $I_{\{r>-n\}}xI{\{|x|\le 1\}}*(\mu-\nu)$, and $I_{\{r>-n\}} |b|\cdot A$ are also locally bounded.  
If all the mentioned processes were bounded, $I_{\{r>-n\}}\cdot S$ would be bounded from below, and the fact that $P$ is a separating measure would give $EI_{\{r>-n\}}\cdot S_T\le 0$, i.e. 
$$
EI_{\{r>-n\}}x^+I_{\{x> 1\}}*\mu_T-EI_{\{r>-n\}}x^-I_{\{x< -1\}})*\mu_T+EI_{\{r>-n\}}b^h\cdot A_T\le 0. 
$$
The first term in this bound is necessarily finite and, therefore,  
$$
EI_{\{r>-n\}}x^+ I_{\{x>1\}}*\nu_T=EI_{\{r>-n\}}x^+I_{\{x>1\}}*\mu_T<\infty. 
$$
This implies, in particular, that $I_{\{r>-n\}}K(|x|I_{\{|x|>1\}})<\infty$ $m$-a.e.
Using the standard localisation procedure, we obtain that $I_{\{r>-n\}}K(|x|I_{\{|x|>1\}})<\infty$ holds $m$-a.e. in the general case where $I_{\{r>-n\}}\cdot S$ is only locally bounded from below. Applying similar arguments to the integrand $I_{\{r>-n\}}I_{D}$ where 
$D\in \cP$, we infer  that  
$$
I_{\{r>-n\}}\big (K\big (xI_{\{|x|>1\}}\big )+b^h\big )\le 0 \quad m\hbox{-a.e.}
$$  
Combining the above it follows that $K\big (|x|I_{\{|x|>1\}}\big )<\infty$ and $K\big (xI_{\{|x|>1\}}\big )+b^h\le 0$ holds $m$-a.e. on the set $\{r>-\infty\}$. Arguing in the same way with the integrand $-I_{\{R<n\}}$ we obtain that $K\big (|x|I_{\{|x|>1\}}\big )<\infty$ and $K\big (xI_{\{|x|>1\}}\big )+b^h\ge 0$ hold $m$-a.e. on the set $\{R<\infty\}$. To recapitulate the previous discussion, modulo a $m$-null subset, the following properties hold: 
\begin{itemize}
\item on $\{r>-\infty\}$  the constant function $1\in \Gamma^\e$ and 
$-b^h \geq \chi(1)=K\big (xI_{\{|x|>1\}}\big )$;       

\item on $\{R<\infty\}$  the constant function $1\in \Gamma^\e$ and 
$-b^h \leq \chi(1)=K\big (xI_{\{|x|>1\}}\big )$. 
\end{itemize}
Therefore, on the intersections of these previous sets it holds that $-b^h=\chi (1)$. The conclusion of the lemma in the case $d=1$ is implied by the following (purely deterministic) claim: \emph{if $R=\infty$, 
then the interval
$\chi (\Gamma^\e)$ is unbounded from above}, along with its  symmetric version involving $r$. This claim is proved in the next paragraph, after which the multi-dimensional case is treated.

Still in the case $d=1$, we prove now that $R=\infty$ implies that then the interval $\chi (\Gamma^\e)$ is unbounded from above. Let $K_n(dx):=I_{\{x>n\}}(x)K(dx)$. For $\gamma>0$ define the set $\cW_{n,\gamma}$ 
of strictly positive functions $W\in C([n,\infty[)$ such that
$W(n)=1$,
$xW(x)\to 0$ as $x\to \infty$, and
$K_n(W)=\gamma$. For any $N>0$, there exists $W_N\in \cW_{n,\gamma}$ such that $K_n(xW)\ge N$. (Indeed, pick  a continuous function $V>0$ such that
$V(n)=1$, $K_n(V)<\infty$ and $K_n(xV)=\infty$.  
Choose $A>n$ such that $K_n(]n,A[)>0$ and $\gamma_1:=K_n\big (VI_{\{x>A\}}\big )<\gamma/2$. 
Take $A'>A$ such that $K_n\big (xVI_{\{A<x\le A'\}}\big )\ge N$. For sufficiently 
large $p$ we have that $\gamma_2:=K_n\big (I_{\{x>A')}V(A'\big )e^{p(A'-x)}\le \gamma/2$. 
Put 
$$W_N:=fI_{\{n<x\le A\}}+VI_{\{A<x\le A'\}}+V(A')e^{p(A'-x)}I_{\{x>A'\}}
$$ 
where 
$f$ is a strictly positive continuous function on $[n,A]$ with $f(n)=1$, $f(A)=V(A)$, and $K_n\big (fI_{\{n<x\le A\}}\big )=\gamma-\gamma_1-\gamma_2$.) Pick now $n\ge 1$ such that
$K({\R}\setminus [-n,n])\le \e/4$. Choose $q>0$ ensuring that 
$$
\delta:=K\big (e^{q(x+n)}I_{\{x<-n\}}\big )<K({\R}\setminus [-n,n]).
$$
 Take
$W_N\in \cW_{n,\gamma}$ with
$\gamma=K({\R}\setminus [-n,n])-\delta$ and $K_n(xW_N)\ge N$.
Then
$$
Y_{N}(x):=e^{q(x+n)}I_{\{x<-n\}}+
I_{\{|x|\le n\}}+W_{N}(x)I_{\{x>n\}}\in \Gamma^\e
$$
and $\chi (Y_{N})\to\infty$ as $N\to \infty$. The claim has been proved.

\smallskip

We continue with the vector case $d \geq 2$ of Lemma \ref{crucial} and show that it  can be reduced to the scalar one. Indeed,
the sets 
$$
\Xi_{\omega,t}:=\chi_{\omega,t}(\Gamma_{\omega,t})+b^h_t(\omega)\subseteq
\R^d
$$ 
are convex and
$\{(\omega,t,x):\ x\in \Xi_{\omega,t}\}\in {\mathcal P} \otimes
{\mathcal B}^d$.
By the measurable version of the separation theorem, there is
a  predictable process $l$ with values in
$\R^d$ such that, outside 
an $m$-negligible set,
$|l_{\omega,t}|=1$ and $l_{\omega,t}x< 0$
for every $x\in \Xi_{\omega,t}$ if $0\notin \Xi_{\omega,t}$, and
$l_{\omega,t}=0$, otherwise.
 We use
the superscript $l$ to denote objects related to the scalar semimartingale  $S^l:=l\cdot S$. It is easily seen that 
$\nu^l(\omega,dt,dx)=K^l_{\omega,t}(dx)dA_t(\omega)$ with
$K_{\omega,t}^l(dx)=\big (K_{\omega,t}l^{-1}_{\omega,t}\big )(dx)$ and
$$B^{l,h}=lb^h\cdot A+K\big (lx(I_{\{|lx|\le 1\}}-I_{\{|x|\le 1\}})\big )\cdot A,
$$ see \cite{JS:87}, IX.5.3;
$P$ is a separating measure for $S^l$. We have proved that
for every fixed $(\omega,t)$ outside of an
$m$-negligible set the equation
$\chi^l_{\omega,t}(Y)=-b^{l,h}_t(\omega)$ has a solution
$Y\in\Gamma^{\e l}_{\omega,t}$. Due to the above relations,
the function $Y(l_{\omega,t}x)$ belongs to $\Gamma^\e_{\omega,t}$ and
solves the equation
$\chi_{\omega,t}(Y(l_{\omega,t}x))=-b^h_t(\omega)$. Thus,
$l=0$, $m$-a.e., completing the proof.
\end{proof}

\section{Boundedness in Probability of Stochastic Exponentials}
\label{boundSE}

For a scalar semimartingale $X$ such that $X_0=1$, $X > 0$ and $X_->0$, define the \emph{stochastic logarithm} as the semimartingale  
$\cL(X):=X_-^{-1}\cdot X$.  Note that  $\Delta \cL(X)>-1$. It easily seen that  $X=\cE(\cL(X))$ and $R=\cL(\cE(R))$ for every semimartingale $R$ with $R_0=0$
and $\Delta R>-1$.

\smallskip
Let $\cR$ be a set of  real-valued  semimartingales $R$ with $R_0= 0$, $\Delta R>-1$. We also define the sets of random variables $\cR_T := \{ R_T: R \in \cR\}$ and $\cE_{T}(\cR) := \{ \cE_{T}(R): R \in \cR\}$.  We say that $\cR$ is \emph{$P$-bounded from above} if the set of random variables $\sup_{s\le T}R_s$ is $P$-bounded. In the same spirit, we define for a set of semimartingales the notions ``\emph{$P$-bounded from below}'' and ``\emph{$P$-bounded}.'' 

Since 
\[
\cE_t(R)=\exp\Bigg\{R_t-\frac 12\langle R^c\rangle_t +\sum_{s\le t}[ \log (1+\Delta R_s)-\Delta R_s] \Bigg\}, \quad R \in \cR,
\]
the inequality $\cE(R)\le e^{R}$ follows; therefore, if $\cR$ is $P$-bounded, so is  $\cE(\cR):=\{\cE(R) : R\in \cR\}$. The converse, in general, may not be true; however, one can prove the following result of independent interest; see \cite[ Lemma A.4]{Kara-Kard}.

\begin{prop} 
\label{expbounded}

Let $\cR$ be a set of semimartingales such that $R_0= 0$, $\Delta R>-1$,  and  $\cE^{-1}(R)$
is a supermartingale for all $R \in \cR$. Put $\cZ:=\cL(\cE^{-1}(\cR))$. Introduce the  following conditions:

\begin{enumerate}
\item[(a)] $\cR$ is $P$-bounded;
\item[(a$'$)]  $\cR$ is $P$-bounded from above;
\item[(a$''$)] $\cR_T$ is $P$-bounded;
\item[(b)] $\cE_{T}(\cR)$ is $P$-bounded; 
\item[(c)] $\cE(\cR)$ is $P$-bounded;
\item[(d)] $\cZ$ is $P$-bounded;
\item[(d$'$)] $\cZ$ is $P$-bounded from below.
\end{enumerate}
Then, it holds that
$$
(a)\Leftrightarrow (a') \Leftrightarrow (a'') \Leftrightarrow  (b)\Leftrightarrow  (c)\Rightarrow  (d)\Leftrightarrow  (d').
$$

\end{prop}

\begin{proof}
The implications $(a)\Rightarrow (a')$, $(a)\Rightarrow (a'')$, $(c)\Rightarrow (b)$, $(d)\Rightarrow (d')$ are trivial, as is $(a'')\Rightarrow (b)$ in virtue of the bound $\cE_T(R)\le e^{R_T}$. 

\smallskip

\noindent $(a')\Rightarrow (a)$. Since $\cE^{-1}(R)$  is a supermartingale when $R\in \cR$, it follows that 
$$
P\Big (\inf_t\cE_t(R)\le n^{-1}\Big )=P\Big (\sup_t\cE^{-1}_t(R)\ge n\Big )\le n^{-1}.
$$
Therefore, $\log \cE(\cR)$ is $P$-bounded from below. Since $R \ge \log \cE(R)$,  the set $\cR$ is always $P$-bounded from below under the assumption of the lemma.  
 
\smallskip

\noindent $(b)\Rightarrow (c)$. If  $(c)$  fails, there are $\e>0$ and  $R^n\in \cR$ such that $\cE_{\tau^n}(R^n)\ge n$
and $P(\tau^n<T)\ge \e$. Using the abbreviation $M^n:=\cE^{-1}(R^n)$ we have, applying Chebyshev's inequality and the supermaringale property, that 
\bean
P(M^n_T \ge n^{-1/2})&=&P(M^n_T/M^n_{\tau^n} \ge n^{-1/2}/M^n_{\tau^n},\ \tau^n<T)\\
&&+P(M^n_T \ge n^{-1/2},\ \tau^n=T)\\
&\le&P(M^n_T/M^n_{\tau^n} \ge n^{1/2})+P(\tau^n=T)\\
&\le & n^{-1/2}+P(\tau^n=T)\le 1-\e/2
\eean
for all $n$ sufficiently large. Thus, $P(\cE_{T}(R^n)>n^{1/2} )\ge \e/2$  in contradiction with $(b)$. 
 
\smallskip

\noindent $(d')\Rightarrow (d)$.  Take arbitrary  $\e>0$.  The set $\cZ$ being bounded from below,  there is  $N_0>0$ such that 
$$
\sup_{Z\in \cZ}P\Big (\inf_t Z_t\le -N+1\Big )\le \e\quad \forall\, N\ge N_0. 
$$
Omitting the dependence on $N$ we define the stopping time 
$$
\tau_Z:=\inf\{t\ge 0:\ Z_t\le -N+1\}.
$$
  Then  $P(\tau_Z<T)\le \e$. Since $\Delta Z>-1$, the local supermartingale  $Z^{\tau_Z}$, being  bounded from below, is a supermartingale, and, by Kolmogorov's inequality (applied to the supermartingale $Z^{\tau_Z}+N\ge 0$)  we have: 
$$
P\Big (\sup_t Z_t\ge N/\e\Big )\le P(\tau_Z<T)+P\Big (\sup_t Z^{\tau_Z}_t\ge N/\e\Big )\le \e+ 1/(1+1/\e)\le 2\e. 
$$  
It follows that $\cZ$ is also $P$-bounded from above, i.e.  $(d)$ holds. 

\smallskip

\noindent $(c)\Rightarrow (d')$. Note that
\beq\label{e1}
\{\cE(R):\  R\in \cR\}=\{\exp\{-\log \cE(Z)\}:\  Z\in \cZ\}
\eeq
 and $\log \cE(Z)\le Z$.
Since  $\cE(\cR)$ is $P$-bounded, so is the set $\{e^{-Z}:\  Z\in \cZ\}$, and $(d')$ holds. 

\smallskip

\noindent $(c)\Rightarrow (a)$. 
For $Z=\cL(1 / \cE(R))$ we have, using the formula for the reciprocal of the stochastic exponential, that 
\beq\label{b1}
Z=-R+\langle R^c\rangle+\sum_{s\le .} \frac{(\Delta R_s)^2}{1+\Delta R_s}. 
\eeq
On the other hand,  
$$
\log \cE(Z) = -\log \cE(R)=-R+\frac 12 \langle R^c\rangle+\sum_{s\le .} (\Delta R_s-\log (1+\Delta R_s)). 
$$
Hence, 
\[
Z-\log \cE(Z)=\frac 12\langle R^c\rangle+\sum_{s\le .}\Big(\log (1+\Delta R_s) - \frac {\Delta R_s}{1+\Delta R_s}\Big). 
\]
We shown already that $(c)$ ensures  that the set $\cZ$ is $P$-bounded, and, in virtue of (\ref{e1}), the set $\{\log \cE(Z):\  Z\in \cZ\}$ is $P$-bounded from below. Therefore, the set of random variables 
$$
\Gamma_1:= \Big\{ \frac 12\langle R^c\rangle_T+\sum_{s\le T}\Big(\log (1+\Delta R_s) - \frac {\Delta R_s}{1+\Delta R_s}\Big),\ R\in \cR\Big\}
$$
is $P$-bounded.  
Property $(a)$ follows from (\ref{b1})  because the sets $\Gamma_1$ and 
$$
\Gamma_2:=\Big\{ \langle R^c\rangle_T+\sum_{s\le T} \frac{(\Delta R_s)^2}{1+\Delta R_s}   ,\ R\in \cR\Big\}
$$
are $P$-bounded simultaneously, the fact requiring some comments. Of course, $P$-boundedness of $\Gamma_2$ implies $P$-boundedness of $\Gamma_1$ because 
$$
\varphi(y):=\log (1+y) - \frac{y}{1+y}\le \psi(y):= \frac{y^2}{1+y}, \qquad y>-1. 
$$
More surprising is the converse implication needed in the proof.  To check it, suppose that  $\Gamma_1$ is $P$-bounded.  Then the set  
$$
\Bigg\{ \langle R^c\rangle_T+\sum_{s\le T} \frac{(\Delta R_s)^2}{1+\Delta R_s}I_{\{\Delta R_s\le 2\}}   ,\ R\in \cR\Bigg\}
$$
is  $P$-bounded due to the inequality $\varphi(y)\ge (1/4) \psi (y)$, valid for $y\in (-1,2]$.  Using the bound  $\sup_{s\le T} \varphi (\Delta R_s)\le \sum_{s\le T} \varphi (\Delta R_s)$, we infer that 
$$
\Big\{\sup_{s\le T}I_{\{\Delta R_s> 2\}}\log (1+\Delta R_s),\ R\in \cR\Big\}
$$ 
is $P$-bounded, implying that also  $\{\sup_{s\le T}\Delta R_s I_{\{\Delta R_s> 2\}},\ R\in \cR\}$ is $P$-bounded. 
Noting that   
\[
\sup_{s\le T}\{(\Delta R_s)^2 I_{\{\Delta R_s> 2\}}\}\sum_{s\le T}\Big(\log (1+\Delta R_s) - \frac {\Delta R_s}{1+\Delta R_s}\Big)I_{\{\Delta R_s> 2\}} \ge \sum_{s\le T} \frac {(\Delta R_s)^2}{1+\Delta R_s} I_{\{\Delta R_s> 2\}}
\]
and the set of random variables in the left-hand side of this inequality when $R$ runs through $\cR$ is $P$-bounded we obtain that the set
$$
\Bigg\{\sum_{s\le T} \frac{(\Delta R_s)^2}{1+\Delta R_s}I_{\{\Delta R_s> 2\}}   ,\ R\in \cR\Bigg\}
$$
is $P$-bounded, and so is $\Gamma_2$. 
\end{proof}

\section{Laws of Large Numbers}
\label{LLN}

The ``classical'' law of large numbers for a locally square integrable martingale $M$ asserts that the  ratio $M_T/(1+\langle M\rangle_T)$ tends a.s. to zero as $T\to \infty$ on the set $\{ \langle M\rangle_\infty = \infty \}$ and tends to 
$M_\infty/(1+\langle M\rangle_\infty)$ on $\{ \langle M\rangle_\infty < \infty \}$. 
Here, we present some simple results in the same spirit for sequences of stochastic integrals with truncated integrands, where the level of truncation, as opposed to the time horizon, tends to infinity. 

Let $J$ be a $d$-dimensional locally square integrable martingale ($J\in \cM^2_{loc}$ in standard notation) with the quadratic characteristics $\langle J\rangle=q\cdot A$, and $H$ be a $d$-dimensional predictable process. 
Set $\Gamma_\infty:=\{|q^{1/2}H|^2\cdot A_T=\infty\}$, and  define the scalar locally square integrable martingale  $M^n:=H^n\cdot J$ where $H^n:=HI_{\{|H|\le n\}}$ for all $n$. Furthermore, set  
$$
L^n:=1+\langle M^n\rangle=1+|q^{1/2}H^n|^2\cdot A.
$$ 

\begin{lemm}
\label{LLN2} The sequence of random variables $M^n_T/L^n_T$ converges to zero  in probability on the set $\Gamma_\infty$, and is $P$-bounded on the set 
$\Omega \setminus \Gamma_\infty$. 
\end{lemm}

\begin{proof}
Put $X^n:=(L^n)^{-3/4}\cdot M^n=(L^n)^{-3/4}H^n\cdot J$. We claim that $(L^n)^{-3/2}\cdot L^n_T\le 2$. Indeed,  using the change of variable 
formula for $F(x):=2 x^{-1/2}$, the finite increment formula,  and the monotonicity of $F'$ we get that
$$
F( L^n_T)-F(1)=F'( L^n)\cdot L_T+\sum_{s\le T}(F( L^n_s)-F( L^n_{s-})-F'( L^n_s)\Delta  L^n_s)\le   F'( L^n)\cdot L^n_T
$$
implying the claimed bound. 
Thus,  
$X^n\in \cM^2$ and, by Doob's inequality, 
\beq\label{doob}
E\sup_{s\le T}|X^{n}_s|^2\le 4E|X^{n}_T|^2=4E(L^n)^{-3/2}\cdot \langle M^n\rangle_T= 4E(L^n)^{-3/2}\cdot L^n_T\le 8.
\eeq
That is, $\sup_{s\le T}|X^{n}_s|$ is bounded in $L^2$, hence, in probability. 

The positive process $U^n:=(L^n)^{3/4}$ is of bounded variation and 
$$
X^n_TU^n_T=X^n_-\cdot U^n_T+U^n\cdot X^n_T=X^n_-\cdot U^n_T+M^n_T.
$$
This implies that 
$$
\frac {M^n_T}{(L^n_T)^{3/4}}=X^n_T-\frac{1}{U^n_T}X^n_-\cdot U^n_T.
$$
It follows that $|M^n_T/(L^n_T)^{3/4}|\le 2\sup_{s\le T}|X^{n}_s|$. Therefore, the sequence of ratios $M^n_T/(L^n_T)^{3/4}$
is $P$-bounded. Since $L^n_T\to 1+|q^{1/2}H|^2\cdot A_T $,  the assertion of the lemma is straightforward. 
\end{proof}

Let now $N^n$ be a sequence of counting processes with compensators of the 
form $\tilde N^n=I_{\{G\le n\}}\cdot \tilde N$ where $\tilde N$ is a predictable increasing c\`adl\`ag process and  
$G\ge 0$ is a predictable process.  Let $\Theta_\infty:=\{\tilde N_T=\infty\}$ and let $R^n_T:=1+\tilde N^n_T$. 
\begin{lemm} 
\label{LLN1}The sequence of random variables $N^n_T/R^n_T\to 1$  in probability  on the set  
$\Theta_\infty$ and is $P$-bounded on the set 
$\Omega \setminus \Theta_\infty$. 
\end{lemm}

\begin{proof}
Put $M^n:=I_{\{G\le n\}}\cdot (N^n-\tilde N^n)$.
Then the process  $M^n\in \cM_{loc}^2$ and $\langle M^n\rangle=(1-\Delta \tilde N^n)\cdot \tilde N^n$.  Exactly in the same way as in the proof of the previous lemma but replacing  $L^n$ by $R^n$ we obtain that $M^n_T/(R^n_T)^{3/4}$
is $P$-bounded (the only change is in (\ref{doob}) where the second equality should be replaced by an inequality). This implies the claim.
\end{proof}

\bibliographystyle{alph}

\end{document}